\documentclass{amsart}
\usepackage{amssymb}
\usepackage{latexsym}
\usepackage{euscript}

\usepackage{todonotes} 
\usepackage{theoremref}
\newcommand{\thlab}[1]{\thlabel{#1}\label{#1 }}
\usepackage{enumerate} 

   \def\sE{{\mathfrak E}}   
   \def\sH{{\mathfrak H}}   
   \def\sK{{\mathfrak K}}   \def\sL{{\mathfrak L}}
\def\sM{{\mathfrak M}}   \def\sN{{\mathfrak N}}   
      \def\sR{{\mathfrak R}}
      
      \def\sX{{\mathfrak X}}
\def\sY{{\mathfrak Y}}   

\def\st{{\mathfrak t}}

      \def\dC{{\mathbb C}}

   \def\dN{{\mathbb N}}

\def\cD{{\EuScript D}}

\def\cM{{\EuScript M}}

\def\bB{{\bf B}} 

\def\wt#1{{{\widetilde #1} }}

\def\bm\chi{\mbox{\boldmath$\chi$}}

\def\ker{{\rm ker\,}}
\def\ran{{\rm ran\,}}
\def\cran{{\rm \overline{ran}\,}}
\def\dom{{\rm dom\,}}
\def\mul{{\rm mul\,}}
\def\cmul{{\rm \overline{mu}l\,}}
\def\cdom{{\rm d\overline{om}\,}}
\def\clos{{\rm clos\,}}

\def\dim{{\rm dim\,}}

\let\xker=\ker \def\ker{{\xker\,}}
\def\spn{{\rm span\,}}

\def\supp{{\rm supp\,}}

\DeclareMathOperator{\hoplus}{\, \widehat \oplus \,}
\DeclareMathOperator{\hominus}{\, \widehat \ominus \,}

\newtheorem{theorem}{Theorem}[section]
\newtheorem{proposition}[theorem]{Proposition}
\newtheorem{corollary}[theorem]{Corollary}
\newtheorem{lemma}[theorem]{Lemma}

\theoremstyle{definition}
\newtheorem{example}[theorem]{Example}
\newtheorem{remark}[theorem]{Remark}
\newtheorem{definition}[theorem]{Definition}

\numberwithin{equation}{section}

\begin{document}

\title[Lebesgue type decompositions]
{Lebesgue type decompositions for linear relations and Ando's uniqueness criterion}
\author[S.~Hassi]{Seppo Hassi}
\author[Z.~Sebesty\'en]{Zolt\'an Sebesty\'en}
\author[H.S.V.~de~Snoo]{Henk de Snoo}

\address{Department of Mathematics and Statistics \\
University of Vaasa \\
P.O. Box 700, 65101 Vaasa \\
Finland}
\email{sha@uwasa.fi}

\address{Department of Applied Analysis \\
E\"otv\"os Lor\'and University\\
P\'azm\'any P\'eter s\'et\'any 1/C \\
1117 Budapest\\
Hungary} \email{sebesty@cs.elte.hu}

\address{Johann Bernoulli Institute for Mathematics and Computer Science \\
University of Groningen \\
P.O. Box 407, 9700 AK Groningen \\
Nederland}
\email{hsvdesnoo@gmail.com}

\dedicatory{To Professor Ando on the occasion of his 85th birthday}

\date{December 29, 2017}
\thanks{The research was partially supported by a grant from the Vilho,
Yrj\"o and Kalle V\"ais\"al\"a Foundation of the Finnish Academy of
Science and Letters.}



\begin{abstract}
A linear operator or, slightly more general, a linear relation
(i.e., a multivalued linear mapping) $T$ from a Hilbert space $\sH$
to a Hilbert space $\sK$ has Lebesgue type decompositions
$T=T_{1}+T_{2}$, where $T_{1}$ is a closable operator and $T_{2}$ is
an operator or relation which is singular. There is one canonical
decomposition, called the Lebesgue decomposition of $T$, whose
closable part is characterized by its maximality among all closable
parts in the sense of domination. All Lebesgue type decompositions
are parametrized, which also leads to necessary and sufficient
conditions for the uniqueness of such decompositions. Similar
results are given for weak Lebesgue type decompositions, where $T_1$
is just an operator without being necessarily closable. Moreover,
closability is characterized in different useful ways. In the
special case of range space relations the above decompositions may
be applied when dealing with pairs of (nonnegative) bounded
operators and nonnegative forms as well as in the classical
framework of positive measures.
\end{abstract}

\maketitle

\section{Introduction}\label{sec1}

In this paper it will be shown that certain notions from measure
theory such as absolute continuity and singularity have analogs for
singlevalued or multivalued linear operators between Hilbert
spaces. Moreover, it will be shown that these analogs have
an extremely simple structure making it possible to obtain
new (and old) results for the decomposition of pairs of operators and forms,
where one operator or form is decomposed with respect to the other
operator or form.  In particular, one may obtain corresponding results in the context
of a pair of measures.

To give a brief review on these fundamental notions, let $X$
be a set, let $\cM$ be a $\sigma$-algebra on it, and let $\lambda$
and $\mu$ be finite positive measures on the $\sigma$-algebra $\cM$.
Then the measure $\lambda$ is said to be \textit{absolutely
continuous} with respect to the measure $\mu$, if
\[
 E \in \cM, \quad \mu(E)=0 \quad \Rightarrow \quad \lambda(E)=0,
\]
and $\lambda$ is said to be \textit{singular}
with respect to $\mu$, if there exists a set $S \in
\cM$ such that
\[
 \lambda(X\setminus S)=0, \quad \mu(S)=0.
\]
A measure $\lambda$ on $(X,\cM)$ has
a unique Lebesgue decomposition into finite positive measures
$\lambda_r$ and $\lambda_s$:
\[
 \lambda=\lambda_r+\lambda_s,
\]
such that $\lambda_r$ is absolutely continuous with respect to $\mu$
and $\lambda_s$ is singular with respect to $\mu$. In the case of
finitely additive measures such decompositions still exist, but they
need not be unique anymore.  Furthermore, the notion of absolute
continuity may be weakened and even then there are similar
decompositions. For some results in this direction, see \cite{STT1,
STT2}.

In fact, it will be shown in the present paper that the context of
linear operators and linear relations from a Hilbert space $\sH$ to a Hilbert space $\sK$
forms the natural framework for such decompositions. Recall that a linear relation
from $\sH$ to $\sK$ is just a linear subspace of the product $\sH \times \sK$,
and that a linear operator is identified with its graph.
Let $T$ be a linear relation from a Hilbert
space $\sH$ to a Hilbert space $\sK$. Then $T$ is said to have an
orthogonal range decomposition
\[
T=T_{1}+T_{2}
\]
with linear relations $T_{1}$ and $T_{2}$ from $\sH$ to $\sK$ if
$\dom T_{1}=\dom T_{2}=\dom T$ and $\ran T_{1} \perp \ran T_{2}$.
Such an orthogonal range decomposition is called a distinguished
orthogonal range decomposition if the relation $T_{1}$ is an
operator. Note that if in this case $T$ is an operator, then
automatically $T_2$ is an operator. In order to cover the Lebesgue
decomposition of measures in this framework the distinguished
orthogonal range decompositions will be studied under the additional
assumption that the relation $T_2$ is singular,
 i.e., the closure of (the graph of) $T_{2}$ is a
product of closed linear subspaces in $\sH$ and $\sK$, respectively.
 One speaks of a \emph{weak Lebesgue type decomposition} if $T_{1}$
is an operator and $T_{2}$ is a singular relation (or operator).
Likewise one speaks of a \emph{Lebesgue type decomposition} if
$T_{1}$ is a regular relation (i.e., a closable operator) and
$T_{2}$ is a singular relation (or operator). It will be shown that
among all weak Lebesgue type decompositions there is precisely one
decomposition, called the \emph{weak Lebesgue decomposition}
\begin{equation}\label{Leb0}
 T=T_{\rm op}+T_{\rm mul},
\end{equation}
where the operator $T_{\rm op}$ admits a certain maximality property
among all operator parts of such decompositions. Similarly, among
all Lebesgue type decompositions there is precisely one
decomposition, called the \emph{Lebesgue decomposition}
\begin{equation}\label{Leb1}
T=T_{\rm reg}+T_{\rm sing},
\end{equation}
where the closable operator $T_{\rm reg}$ admits a maximality
property among all closable operator parts of such decompositions.

The notion of Lebesgue decomposition for linear operators goes back
to Jorgensen \cite{J} and was also considered by \^Ota \cite{Ota1,
Ota3}. The Lebesgue decomposition of Jorgensen was put in the
context of linear relations in \cite{HSeSnSz} and \cite{HSnSz}. The
present work is a continuation of these papers and of parts in
\cite{HSeSnLeb}. The main interest is in the new concepts of
Lebesgue type and weak Lebesgue type decompositions in the general
setting of linear relations. Since the weak Lebesgue decomposition
\eqref{Leb0} and the Lebesgue decomposition \eqref{Leb1} are
well-defined, the question is how they relate to the other
 weak Lebesgue type and Lebesgue
type decompositions. Among the main results are descriptions of all
possible Lebesgue type and weak Lebesgue type decompositions of $T$,
the characteristic properties of the operator $T_{\rm reg}$ in the
Lebesgue decomposition \eqref{Leb1}, and the operator $T_{\rm op}$
in the weak Lebesgue decomposition \eqref{Leb0} of $T$, as well as
uniqueness theorems offering some necessary and sufficient
conditions for the (weak) Lebesgue decomposition to be the only
(weak) Lebesgue type decomposition of $T$. The uniqueness result was
inspired by a similar result of Ando in a more special situation,
cf. \cite{An}. On the other hand, the maximality properties that
will be established on the components $T_{\rm op}$ and $T_{\rm reg}$
in \eqref{Leb0} and \eqref{Leb1} rely on the concept of
\emph{domination} for (unbounded) operators and relations as
developed in \cite{HSn2015}.

The motivation for the present study of Lebesgue and weak Lebesgue
type decompositions of linear relations comes from various
applications which can be embedded properly in the present general
framework. Only a brief discussion is appropriate here; for the
details, see \cite{HSn2018}. First of all it will be convenient to
restrict the general class of linear relations as linear subspaces
of the product space $\sH \times \sK$. Recall that a linear subspace
of a Hilbert space is called an operator range if it coincides with
the range of a bounded linear operator between Hilbert spaces. The
notion of an operator range extends the notion of a closed linear
subspace. In particular, a linear relation $T$ from a Hilbert space
$\sH$ to a Hilbert space $\sK$ is said to be a \textit{range space
relation}  if $T=\ran C$ for some bounded operator
$C\in\mathbf{B}(\sE, \sH\times\sK)$, where $\sE$ is a Hilbert space;
the notion of a range space relation extends the notion of a closed
relation. In fact, the operator $C$ induces the following
representation for $T$:
\begin{equation}\label{ToperRange}
 T=\left\{\, Ch\in\sH\times \sK:\, h \in \sE\right\}=\left\{\{Ah,Bh\} :\, h \in \sE\right\},
\end{equation}
involving the component mappings $A \in \mathbf{B}(\sE, \sH)$ and $B
\in \mathbf{B}(\sE, \sK)$.  Note that as a direct consequence of
this representation also $\dom T$ and $\ran T$ are operator ranges
in $\sH$ and $\sK$, respectively. The (weak) Lebesgue type
decompositions of a relation $T$ take an interesting form when $T$
is a range space relation. The results of the present paper are now
reflected in the properties of the pair of bounded operators $A$ and
$B$. In particular, the operator $B$ can be decomposed via an
operator sum decomposition
\[
 B=B_{1}+B_{2},
\]
where $B_{1}$ is almost dominated by $A$ and $B_{2}$ is singular
with respect to $A$. All such decompositions can be parametrized and
there is a uniqueness result. Under some additional conditions on
$A$ and $B$, such decompositions have been investigated in Izumino
\cite{I89a}. Furthermore, it should be remarked that in this
situation of a range space relation there is now a Radon-Nikodym
derivative. When $T$ is represented by $A \in \mathbf{B}(\sE, \sH)$
and $B \in \mathbf{B}(\sE, \sK)$ as above and $B$ has the
corresponding Lebesgue type decomposition, then the Radon Nikodym
derivative is the operator expressing the operator $B_1$ in terms of
the operator $A$. This can be seen as the proper analog of the
corresponding notion in measure theory; cf. \cite{HSn2018}.

For a more special application consider the case of a pair bounded
nonnegative operators $A$ and $B$, which was first studied by Ando
\cite{An}. This particular situation can be set up in the same
framework in an analogous way.
 Hence also the results concerning
Ando's decompositions for pairs of nonnegative operators
are direct consequences of the results in the present paper.
The necessary and sufficient uniqueness conditions as formulated by Ando follow
from the general relation case.
Furthermore, after associating appropriate Hilbert spaces to a pair
of nonnegative forms, the results on Lebesgue type decomposition for
pairs of nonnegative forms in \cite{HSeSnLeb} can be seen as a
special case; see \cite{S3} for Lebesgue decompositions of forms.
 Of course, Lebesgue decomposition
results for measures and the corresponding Radon-Nikodym derivatives
can be recovered from the corresponding results
on linear relations; cf. \cite{HSeSnLeb,STT1,STT2}. For a complete
treatment of these results and applications see \cite{HSn2018},
where also a  more complete list of references may be found.

Here is an overview of the paper. Section \ref{sec2} is a
preliminary section about linear relations and the notions of
regular and singular relations. In Section \ref{sec3} orthogonal
range decompositions of linear relations are defined and a criterion
is given so that one of the summands is a (regular) operator or a
singular relation. Section \ref{sec4} contains on overview on the
various aspects of the Lebesgue decomposition into its regular and
singular parts. In Section \ref{sec5} Lebesgue type decompositions
of relations are introduced  and parametrizations of all such
decompositions are studied including some criteria leading to
one-to-one parametrizations. Necessary and sufficient conditions for
a Lebesgue type decomposition to be unique can be found in Section
\ref{sec6}. These uniqueness results can be seen as a straighforward
consequence of the given parametrization. A treatment of weak
Lebesgue type decompositions with an associated uniqueness criterion
is given in Section \ref{sec7}. The notion of domination for the
present context of unbounded operators and linear relations plays a
crucial role in Section \ref{sec8}. This is a key notion for
deriving the maximality property for the operator parts $T_{\rm
reg}$ and $T_{\rm op}$ of a relation $T$; they single out optimality
of the decompositions \eqref{Leb0} and \eqref{Leb1}. Furthermore,
the notion of domination is used to establish a general criterion
for the closability of an operator. This section also contains a
metric criterion for an operator to be closable.

\section{Preliminaries}\label{sec2}

This section contains the necessary ingredients about relations
which are needed in establishing the basic decompositions that will
be investigated in later sections of the paper. First conditions
will be given so that a dense subspace of a Hilbert space has a
dense intersection with a given closed subspace.

\subsection{A denseness result}
Recall the definition of an operator range which is a special
subspace of a Hilbert space; cf. \cite{FW}.

\begin{definition}
A subspace $\sR$ of a Hilbert space $\sH$ is said to be an operator
range (for short, a range space) if there exists a Hilbert space
$\sE$ and an operator $B \in \bB(\sE,\sH)$ such that $\sR=\ran B$.
\end{definition}

The next lemma is involved with the following situation. Let $\sM$
be a closed subspace of a Hilbert space $\sH$ and let $\sR \subset
\sM$ be a subspace which is dense in $\sM$. Note that for any closed
subspace $\sL$ of $\sH$ it follows that
\[
 P_{\sL}\, \sR \mbox{ dense in } P_{\sL}\,\sM,
\]
so that these two subspaces have the same closure.
Here $P_\sL$ stands for the orthogonal projection from $\sH$ onto $\sL$.
The question is when the subspace
$\sR \cap \ker P_\sL$ is dense in the subspace $\sM \cap \ker P_\sL$.

\begin{lemma}\label{nneeww}
Let $\sM$ and $\sL$ be closed subspaces of the Hilbert space $\sH$
and let $\sR \subset \sM$ be a dense subspace of $\sM$. Assume, in
addition, that $\sR$ is an operator range, then the identity
\begin{equation}\label{PdomT*0}
P_{\sL}\,\sR=\sL
\end{equation}
implies
\begin{equation}\label{PdomT*0+}
\sM \cap \ker P_{\sL}=\clos( \sR \cap \ker P_{\sL}).
\end{equation}
Furthermore, if the subspace $\sL$ is finite-dimensional then every
dense subspace $\sR$ of $\sM$ for which $P_{\sL}\,\sR$ is dense in
$\sL$ satisfies the identities \eqref{PdomT*0} and \eqref{PdomT*0+}.
\end{lemma}

\begin{proof}
The following decomposition is immediate for any subspace $\sR$
which satisfies the condition \eqref{PdomT*0}:
\begin{equation}\label{EEN}
\sH={P_\sL}^{-1}(\sL)=(P_{\sL})^{-1} (P_{\sL} \,\sR) =\ker P_\sL + \sR.
\end{equation}
If, in addition, $\sR$ is an operator range then it is known that
there exist a closed subspace $\sR_{0} \subset \sR$ and a closed
subspace $\sX_{0} \subset \ker P_{\sL}$ such that $\sH=\sX_{0}
+\sR_{0}$, a direct sum; see e.g. \cite[Theorem~2.4]{FW}. Thus in
this case \eqref{EEN} leads to the existence of a closed subspace
$\sR_0 \subset \sR$ such that
\begin{equation}\label{TWEE}
 \sH=\ker P_\sL + \sR_{0},
\end{equation}
where the sum is not necessarily direct.
It is clear that
\begin{equation}\label{DRIE}
 \sR_{0}=(\ker P_\sL\cap \sR_{0})  \oplus  ( \sR_{0}\ominus (\sR_{0} \cap \ker P_\sL)).
\end{equation}
Combining \eqref{TWEE} and \eqref{DRIE} one gets the sum decomposition
\begin{equation}\label{VIER}
 \sH=\ker P_\sL \,+\, ( \sR_{0}\ominus (\sR_{0} \cap \ker P_\sL)),
\end{equation}
and it is clear that this decomposition is direct. Introduce the
closed subspace
\[
\sR_{1}:=\sR_{0}\ominus (\sR_{0} \cap \ker P_\sL) \subset  \sR_{0}.
\]
Thus it follows from \eqref{VIER} that the closed subspace $\sR_1$ satisfies
\begin{equation}\label{VIJF}
  \sH=\ker P_\sL \,\dot{+}\, \sR_{1}, \quad \sR_1\subset \sR.
\end{equation}
Next it is shown that the decomposition \eqref{VIJF} implies the
property \eqref{PdomT*0+}. Indeed, due to $\sR_1\subset \sR$ it is a consequence
of \eqref{VIJF} that $\sR$ has the decomposition
\begin{equation}\label{Rdec}
 \sR= (\sR \cap \ker P_{\sL} ) \,\dot{+}\,
 \sR_{1}.
\end{equation}
By the positivity of the angle between the subspaces $\ker P_\sL$
and $\sR_1$ in \eqref{VIJF}, the identity \eqref{Rdec} implies that
\[
 \sM=\clos \sR
 =\overline{(\sR \cap \ker P_\sL  ) } \,\dot{+}\, \sR_1
\]
and, consequently, due to the definition of $\sR_{1}\,(\subset\sM)$,
\[
 \sM \cap \ker P_\sL =\overline{\sR\cap \ker P_\sL},
\]
which is the property \eqref{PdomT*0+}.

Now assume that $\sL$ is finite-dimensional and that $P_\sL \sR$ is
dense in $\sL$. Since $\sL$ is finite-dimensional the subspace
$P_\sL \sR$ is closed and thus \eqref{PdomT*0} holds. On the other
hand, when $\dim \sL <\infty$ then one can immediately find a closed
subspace $\sR_1 \subset \sR$ with $\dim \sR_1=\dim \sL$ such that
$P_\sL(\sR_1)=\sL$; this gives \eqref{VIJF} directly. As shown above
this implies the identity \eqref{PdomT*0+}.
\end{proof}

The above lemma extends the following known result; cf. e.g.
\cite{GK}.

\begin{corollary}\label{GKcor}
Let $\sN$ be a closed subspace of a Hilbert space $\sH$ and assume
that its orthogonal complement $\sL=\sH\ominus \sN$ is
finite-dimensional. Then every linear subspace $\sR$ which is dense
in $\sH$ has a dense intersection with $\sN$:
\[
\sN =\clos( \sR \cap \sN).
\]
\end{corollary}

\begin{proof}
Notice that if $\sN$ is a closed subspace of $\sH$, then one may
apply Lemma~\ref{nneeww} with $\sM=\sH$ and $\sL=\sH\ominus \sN$.
Then clearly $P_\sL\,\sM=\sL$ and $P_{\sL}\,\sR$ is dense in $\sL$.
Now by the second part of Lemma~\ref{nneeww} the equalities
\eqref{PdomT*0} and \eqref{PdomT*0} hold. Since here $\ker
P_{\sL}=\sN$ and $\sR \cap \ker P_{\sL}=\sR \cap \sN$, the proof is
complete.
\end{proof}

Lemma~\ref{nneeww} will be used in Section \ref{sec5} to produce
Lebesgue type decompositions for unbounded operators and linear
relations $T$ which differ from the Lebesgue decomposition of $T$ in
\eqref{Leb1} when $\dom T^*$ is not closed.

\subsection{Linear relations}\label{sec2.2}

A general treatment of linear relations as an extension of the
notion of linear operator goes back to \cite{Arens}. Here a few
preliminary facts are recalled; for more see for instance
\cite{HSnSz}. A linear relation (or relation for short) $T$ from a
Hilbert space $\sH$ to a Hilbert space $\sK$ is a linear subspace of
the product $\sH \times \sK$. Its domain, range, kernel, and
multivalued part are denoted by $\dom T$, $\ran T$, $\ker T$, and
$\mul T$.  A relation is (the graph of) an operator if and only if
$\mul T=\{0\}$. The \textit{inverse} $T^{-1}$ of a linear relation
$T$ is defined as $T^{-1}=\{\,\{g,f\} :\,\{f,g\} \in T\,\}$. A
relation $T$ from a Hilbert space $\sH$ to a Hilbert space $\sK$  is
said to be \textit{closed} if it is closed as a subspace of the
product space $\sH \times \sK$. For a relation $T$ the adjoint $T^*$
is given by
\begin{equation}\label{adjo}
T^*=JT^\perp=(JT)^{\perp},
\end{equation}
where $J\{f,f'\}=\{f',-f\}$, $\{f,f'\} \in \sH \times \sK$; hence
$T^*$ is automatically a closed linear relation from
 $\sK$ to  $\sH$. Thus
\begin{equation}\label{adjoo}
T^*=\{\,\{h,k\} \in \sK \times \sH :\, (g,h)=(f,k) \mbox{ for all } \{f,g\} \in T\,\}.
\end{equation}
The definition of the adjoint leads to $T^*=(JT)^\perp$,
so that $T^{**}=T^{\perp \perp}$ and
\begin{equation}\label{grij0}
 \overline{T}=T^{**},
\end{equation}
where $\overline{T}$ is the closure of the linear relation $T$.
Hence a relation $T$ is closed  precisely when  $T^{**}=T$.
 It is straightforward to check the following   identities
\[
 (\dom T)^{\perp}=\mul T^{*}, \quad  (\dom T^{*})^{\perp}=\mul T^{**},
 \quad  (\dom T^{**})^{\perp}=\mul T^{*},
\]
\[
 (\ran T)^{\perp}=\ker T^{*}, \quad  (\ran T^{*})^{\perp}=\ker T^{**},
 \quad  (\ran T^{**})^{\perp}=\ker T^{*}.
\]
Note that \eqref{grij0} leads to the identity $\sH\times\sK=\overline{T} \oplus
T^\perp=T^{**} \oplus JT^*$, so that there are also
nonorthogonal decompositions of the Hilbert spaces:
\begin{equation}\label{eqq1}
 \sH=\dom T^{**}+\ran T^*,\quad \sK=\dom T^*+\ran T^{**}.
\end{equation}
 A general principle shows that $ \dom T^{**} \subset \sH$
and $\dom T^{*} \subset \sK$ are simultaneously closed,
and that $ \ran T^{**} \subset \sK$
and $\ran T^{*} \subset \sH$ are simultaneously closed.

For relations $T_1$ and $T_2$  from $\sH$ to $\sK$ the \textit{sum}
of $T_1$ and $T_2$ is a relation from $\sH$ to $\sK$ defined by
\begin{equation}\label{Tdecomp00}
  T_1+T_2=\{\,\{f,h+k\} :\, \{f,h\} \in T_1, \,\, \{f,k\} \in T_2\,\}.
\end{equation}
If $T=T_{1}+T_{2}$ then clearly
$\dom T=\dom T_{1} \cap \dom T_{2}$, while
\begin{equation}\label{rm1}
  \ran T \subset \ran T_{1} + \ran T_{2}
  \quad \mbox{and} \quad \mul T=\mul T_{1} + \mul T_{2}.
\end{equation}
When $T=T_{1}+T_{2}$, it will be no restriction to consider $T_1$
and $T_2$ together with $T$ on their joint domain $\dom T$.

Let $T_1$ be a relation from a Hilbert space $\sM$ to a Hilbert space $\sK$ and
let $T_2$ be a relation  from the Hilbert space $\sH$ to $\sM$.
Then the \textit{product} $T_1T_2$ is a relation from $\sH$ to $\sK$ defined  by
\begin{equation}\label{prod0}
 T_1T_2=\{\,\{f,f'\} \in \sH \times \sK :\, \{f,\varphi\} \in T_2, \, \{\varphi,f'\} \in T_1,
 \mbox{ for some }  \varphi \in \sM\,\}.
\end{equation}
If $T_{1}$ is an operator then the above definition \eqref{prod0}
can be written as
\begin{equation}\label{prod0+}
T_{1}T_{2}=\{ \{f,T_{1}\varphi \} :\, \{f, \varphi\} \in T_{2}\},
\end{equation}
while if $T_{2}$ in an operator, then \eqref{prod0} can be written
as
\begin{equation}\label{prod0++}
 T_{1}T_{2}=\{ \{f,f'\} : \{T_{2}f, f'\}  \in T_{1}, \; f\in \dom T_{2}\}.
\end{equation}
For the adjoint of the product one has
 \begin{equation}\label{prod1}
 T_2^*T_1^* \subset (T_1T_2)^*,
\end{equation}
with equality when $T_1 \in \bB(\sM,\sK)$, the class of all bounded
everywhere defined operators from $\sM$ to $\sK$.

Assume that $T$ is a closed linear relation from a Hilbert space
$\sH$ to a Hilbert space $\sK$. Then the linear subspace $\mul T$ is
closed and let $P$ be the orthogonal projection from $\sK$ onto
$\mul T$. Define the operator part $T_{\rm s}$ by
\[
 T_{\rm s}=(I-P)T=\{\{f, (I-P)f'\} :\, \{f,f'\} \in T\}.
\]
Then it is clear that $T_{\rm s} \subset T$ and that
\[
 T=T_{\rm s} \hoplus (\{0\} \times \mul T),
\]
which is a componentwise orthogonal sum of the graph of the closed linear operator $T_{\rm s}$
and the purely multivalued closed relation $\{0\} \times \mul T$. The operator $T_{\rm s}$ is called
the orthogonal operator part of $T$; cf. \cite{Arens}.

Assume that $T$ is a closed operator from $\sH$ to $\sK$. Then
$T^*T$ is a nonnegative relation in $\sH$. To see this let $\{f,f'\} \in T^*T$; the by the definition of the
product one sees that  $\{f,h\} \in T$ and $\{h,f'\} \in T^*$ for some $h \in \sK$. This leads to
\[
 (f',f)=\|h\|^2 \geq 0,
\]
which means that the relation $T^*T$ is nonnegative. In order to
show that $T^*T$ is selfadjoint, it suffices to show that $\ran
(T^*T+I)=\sH$. Let $h \in \sH$, then there is a unique decomposition
\[
 \{h,0\} =\{\varphi,\varphi'\}+\{\psi,\psi'\}, \quad
 \{\varphi,\varphi'\} \in T, \quad \{\psi,\psi'\} \in
 T^\perp=JT^*,
\]
since $\sH^2=T \oplus T^\perp$. Hence
\[
 h=\varphi+\psi, \quad \varphi'+\psi'=0,
\]
which leads to $\{\psi,\psi'\}=\{\psi,-\varphi'\}\in JT^*$ and
$\{\varphi',\psi\} \in T^*$. Therefore, $\{\varphi,\psi\} \in
T^*T$ and
\[
 \{\varphi,h\}=\{\varphi,\varphi+\psi\} \in T^*T+I,
\]
so that $h \in \ran (T^*T+I)$. Thus $\ran (T^*T+I)=\sH$.
Hence it follows that $T^*T$ is a nonnegative selfadjoint relation in $\sH$.
Likewise it is not difficult to see that $\mul T^*T=\mul T^*$. This means that
\[
 T^*T=(T^*T)_{\rm s} \hoplus (\{0\} \times \mul T^*),
\]
so that $(T^*T)_{\rm s}$ is a densely defined nonnegative selfadjoint operator
in the Hilbert space
\[
\sH \ominus \mul T^*=\cdom T^{**}=\cdom T,
\]
which follows from the above identities after \eqref{grij0}. Therefore
$(T^*T)_{\rm s}$ has the representation
\[
(T^*T)_{\rm{s}}^{1/2}=\int_0^\infty \lambda\,dE_\lambda,
\]
where $E(\lambda)$ is a family of orthogonal projections in the Hilbert space $\cdom T$.

\subsection{Linear relations and orthogonal projections}

The following result shows that the product $QT$ of a relation $T$ and an
orthogonal projection can be used to decompose $T$ when
$\mul T$ is invariant under $Q$.

\begin{lemma}\label{orth+}
Let $T$ be a  relation from the Hilbert space $\sH$ to the Hilbert
space $\sK$ and let $Q$ be an orthogonal projection
from $\sK$ onto some
closed subspace of $\sK$. Then
\begin{equation}\label{incll}
 T \subset (I-Q)T+QT,
\end{equation}
and, moreover,
\begin{equation}\label{incl0}
 T = (I-Q)T+QT \quad \Leftrightarrow \quad Q \,\mul T \subset \mul T.
\end{equation}
In this case, $\mul T=(I-Q)\mul T + Q\mul T$.
\end{lemma}

\begin{proof}
The inclusion \eqref{incll} is clear, for if $\{f,g\} \in T$, then
\[
 \{f,g\}=\{f, (I-Q)g+Qg\}, \quad \{f,(I-Q)g\} \in (I-Q)T,
 \quad  \{f,Qg\} \in QT;
\]
cf. \eqref{prod0+} and \eqref{Tdecomp00}.
Now the characterization of the identity \eqref{incl0} will be shown.

($\Rightarrow$) Assume that $(I-Q)T+QT \subset T$. If $\{0,g\} \in
T$, then it follows that $\{0,Qg\} \in (I-Q)T+QT \subset T$, so that
$Qg \in \mul T$. Hence $Q \,\mul T \subset \mul T$.

($\Leftarrow$) Let $\{f,g\} \in T$ and $\{f,g'\} \in T$, then
$\{0,g-g'\} \in T$. Hence $\{0, Q(g-g')\} \in T$ and, therefore
\[
 \{f, (I-Q)g+Qg'\}=\{f,g\} -\{0, Q(g-g')\} \in T,
\]
so that $(I-Q)T+QT \subset T$.

The last statement is obtained by applying the identity in \eqref{rm1}.
\end{proof}

Notice that if $T$ itself is (the graph of) an operator, then also
$QT$ and $(I-Q)T$ are (the graphs of) operators. The observations in
the next lemma will be helpful in the rest of the paper.

\begin{lemma}\label{PT}
Let $T$ be a relation from $\sH$ to $\sK$
and let $Q$ be an orthogonal projection in $\sK$.
The adjoint of the relation $QT$ is given by
\begin{equation}\label{PT00}
(QT)^{*}=T^{*}Q,
\end{equation}
and its domain is given by
\begin{equation}\label{PT0}
 \dom T^*Q=\ker Q \oplus (\dom T^* \cap \ran Q),
\end{equation}
while its kernel is given by
\begin{equation}\label{PT0k}
 \ker T^*Q=\ker Q \oplus (\ker T^* \cap \ran Q).
\end{equation}
In particular, $T^{*}Q$ is densely defined in $\sK$ if and only if
\[
 \clos (\dom T^* \cap \ran Q) =\ran Q,
\]
and $\dom T^{*}Q =\sK$ if and only if
\[
\ran Q \subset \dom T^{*}.
\]
\end{lemma}

\begin{proof}
Since $Q$ is an everywhere defined bounded operator, the adjoint of
the relation $QT$ is given by $T^{*}Q$; see \eqref{prod1}.

Now \eqref{PT0} and \eqref{PT0k} will be shown.
Let $f \in \dom T^{*}Q$, then there exists $g \in \sK$ with
$\{f,g\} \in T^{*}Q$ or $\{Qf,g\} \in T^{*}$. This shows via
\[
 f=(I-Q)f+Qf, \quad (I-Q)f \in \ker Q, \quad Qf \in \dom T^{*},
\]
that $f \in \ker Q \oplus (\ran Q \cap \dom T^{*})$. Hence the
left-hand side of \eqref{PT0} is contained in the right-hand side.
 The reverse inclusion follows from the straightforward inclusions
$\ker Q \subset \dom T^{*}Q$ and $\dom T^{*} \cap \ran Q \subset
\dom T^{*}Q$. Hence \eqref{PT0} is clear and the proof of
\eqref{PT0k} is completely similar.

The last two statements are clear from \eqref{PT0}.
\end{proof}

\subsection{Regular and singular relations}

Let $T$ be a relation from a Hilbert space $\sH$ to a Hilbert space $\sK$.
Observe that the trivial inclusion $T \subset T^{**}$ leads to
\begin{equation}\label{mul}
\mul T \subset \cmul T \subset \mul T^{**}.
\end{equation}
It is clear that $T$ is an operator precisely when $\mul T=\{0\}$ or, equivalently,
$\cmul T=\{0\}$;
however in this case the closed linear subspace $\mul T^{**}$ need not be trivial.

\begin{definition}
Let $T$ be a relation from a Hilbert space $\sH$ to a Hilbert space $\sK$.
Then the relation $T$ is called \textit{regular} (or closable) if its closure $T^{**}$
is the graph of an operator.
The relation $T$ is called \textit{singular} if its closure $T^{**}$ is equal to
the product of closed linear subspaces in $\sH$ and $\sK$.
\end{definition}

The next result contains some characterizations and specifications
for regularity; see also \cite[Proposition~3.1]{HSeSnSz},
\cite[Propositions~3.4,~3.5]{HSnSz} and references therein.

\begin{proposition}\label{stoch0}
Let $T$ be a relation from a Hilbert space $\sH$ to a Hilbert
space $\sK$. Then the following statements are equivalent:
\begin{enumerate}[ $ (\rm i)$]
\item $T$ is regular, i.e., $T^{**}$ is an operator;
\item  $\ran T^{**} \subset \cdom T^*$;
\item $\cdom T^*=\sK$.
\end{enumerate}
Moreover, the following statements are equivalent:
\begin{enumerate}[ $(\rm i)$]
\setcounter{enumi}{3}
\item $T$ is strongly regular, i.e., $T^{**}$ is a bounded operator;
\item $\ran T^{**} \subset \dom T^*$;
\item $\dom T^*=\sK$.
\end{enumerate}
Finally, the following statements are equivalent:
\begin{enumerate}[ $ (\rm i)$]
\setcounter{enumi}{6}
\item $T^{**} \in \bB(\sH,\sK)$;
\item $\ran T^{**} \subset \dom T^*$, $\ran T^* \subset \dom T^{**}$;
\item $\dom T^*=\sK$, $\dom T^{**}=\sH$.
\end{enumerate}
\end{proposition}

\begin{proof}
It is clear that  (i) $\Leftrightarrow$ (iii) and that (iii) $\Rightarrow$ (ii).
For (ii) $\Rightarrow$ (iii) note that
$\mul T^{**} \subset \cdom T^*=(\mul T^{**})^{\perp}$
implies $\mul T^{**}=\{0\}$.

For (iv) $\Rightarrow$ (vi) note that
$\dom T^{**}$ is closed implies $\dom T^*$ is closed
while $\mul T^{**}=\{0\}$ implies $\cdom T^{*}=\sK$.
For (vi) $\Rightarrow$ (iv)
note that $\dom T^{**}$ is closed and
that $T^{**}$ is an operator.
Then use the closed graph theorem.
It is clear that (vi) $\Rightarrow$ (v),
while (v) $\Rightarrow$ (vi) follows from \eqref{eqq1}.

Note that (vii) $\Leftrightarrow$ (ix) follows
from (iv) $\Leftrightarrow$ (vi).
Moreover (ix) $\Rightarrow$ (viii) is trivial,
while (viii) $\Rightarrow$ (ix) follows from \eqref{eqq1}.
\end{proof}

By definition a relation $T$ from $\sH$ to $\sK$ is singular if and
only if $T^{**}=\sX \times \sY$ with closed linear subspaces $\sX
\subset \sH$ and $\sY \subset \sK$. In particular, any relation from
$\sH$ to $\sK$ which is a product of (not necessarily closed) linear
subspaces is singular. Furthermore, it is clear $T$ and $T^{-1}$ are
simultaneously singular. The next result contains some central
characterizations of singular relations; see also
\cite[Proposition~3.2]{HSeSnSz}, \cite[Proposition~3.3]{HSnSz}.
Again a short proof is given for completeness.

\begin{proposition}\label{Tsinglemma0}
Let $T$ be a relation from a Hilbert space $\sH$ to a Hilbert
space $\sK$. Then the following statements are equivalent:
\begin{enumerate}[ $(\rm i)$]
\item $T$ is singular, i.e., $T^{**} =\dom T^{**} \times \ran T^{**}$;
\item $\dom T^{**}=\ker T^{**}$;
\item $\ran T^{**}=\mul T^{**}$;
\item $T^{*}$ is singular, i.e., $T^{*} =\dom T^{*} \times \ran T^{*}$;
\item $\dom T^{*}=\ker T^{*}$;
\item $\ran T^{*}=\mul T^{*}$.
\end{enumerate}
\end{proposition}

\begin{proof}
It is clear that (i) $\Rightarrow$ (ii), (iii).  For (iii) $\Rightarrow$ (i)
let $\{f,g\} \in  \dom T^{**} \times \ran T^{**}$.
Then there exists $f' \in \sK$ such that $\{f,f'\} \in T^{**}$ and thus
$g-f' \in \ran T^{**}=\mul T^{**}$, so that
$\{f,g\} =\{f,f'\}+\{0,g-f'\} \in T^{**}$.
Hence $\dom T^{**} \times \ran T^{**} \subset T^{**}$
and the reverse inclusion is obvious. For (ii) $\Rightarrow$ (i)
a similar argument can be used.
Furthermore, it is clear that $T^{**} =\sX \times \sY$, with
closed linear subspaces $\sX \subset \sH$ and $\sY \subset \sK$,
if and only if $T^{*}=\sY^{\perp} \times \sX^{\perp}$.
Hence   (iv), (v), and (vi) follow by applying (i), (ii), and (iii)
with $T^{**}$ replaced by $T^{*}$.
\end{proof}

In particular, $T$ and $T^*$ are simultaneously singular. Notice
also that for a singular $T$ the sets $\dom T^*$, $\ran T^*$, $\dom
T^{**}$, and $\ran T^{**}$ are necessarily closed subspaces. A
relation $T$ is simultaneously regular and singular precisely when
$T^{**}=\dom T^{**} \times \{0\}$, i.e., when $T^{**}$ is the zero
operator on its domain. This statement is equivalent to $T^{*}=\sK
\times \ran T^{*}$.  For further results on regular and singular
relations and their connections to certain decomposability
properties of $T$, see \cite{HSnSz}.

\subsection{Existence of singular operators or relations} The existence of a singular
operator or relation will be illustrated by means of a simple example
going back to  J.~Brasche; cf. \cite[pp. 314--315]{Wer}.
Related examples can be found in the literature,
see for instance \cite[pp. 447--448]{Stone}, \cite[pp. 72--73]{Weid},
\cite[p. 351]{Wer}.

\begin{example}[Point evaluations]\label{L-decom0+}
 Let $\sH=L^{2}[0, \infty)$, let $\sK$ be a Hilbert space with
an orthonormal basis $(e_{n})$,
and let $(x_{n})$  be a strictly
increasing unbounded sequence in $[0,\infty)$.
Let $\cD$ stand for the continuous functions with compact support
on $[0,\infty)$ and define $T$ by
\begin{equation}\label{infinite}
 T=\left\{\,\left\{f, {\sum}^{\infty}_{n=1} f(x_{n})e_{n} \right\} :\, f \in \cD\,\right\}.
\end{equation}
The sum ${\sum}^{\infty}_{n=1} f(x_{n})e_{n}$ is actually a finite
sum since the function $f$ has compact support. Hence $T$ is a
well-defined operator from $\sH$ to $\sK$ with dense domain in
$\sH$. Note that $\left\{h,k\right\} \in T^{*}$ if and only if
\begin{equation}\label{adjEx}
\left(\sum^{\infty}_{n=1}f(x_{n})e_{n},h\right)=\left(f,k\right)
\end{equation}
for all $f \in \cD$.
Let $n_0 \in \dN$ be arbitrary
and choose a nontrivial interval $I_{n_0} \subset [0,\infty)$
with $x_{n_0} \in I_{n_0}$
which does not contain the other points $x_j$, $j \neq n_0$.
Let $(f_m)$ in $\cD$ be a sequence with
\[
 \supp(f_{m}) \subset I_{n_0},
 \quad f_{m}(x_{n_0})=1,
 \quad
 \left\| f_{m} \right\|_{L^{2}[0,\infty)} \rightarrow 0  \mbox{ as } m \to \infty.
\]
Now using \eqref{adjEx} shows that
\[
 |\left(e_{n_0},h\right)| = |\left(f_m,k\right)| \leq
 \|f_m\|_{L^{2}[0, \infty)}\|\,\|k\|_{L^{2}[0, \infty)}
\]
and taking limits as $m\to \infty$ leads to $(e_{n_0},h)=0$. Since
$n_{0}$ is arbitrary this leads to $h=0$ by Parseval's identity. The
identity $(f,k)=0$ for all $f \in \cD$ then gives $k=0$. It follows
that
\[
T^{*}=\left\{0\right\} \times \left\{0\right\}, \quad
T^{**}=\sH \times \sK.
\]
Thus $T$ is a densely defined singular operator.

For a finite number of point evaluations let $(e_{n})$, $1 \le n \le
N$, be a finite orthonormal sequence in the Hilbert space $\sK$ and
let $x_{1}, \dots, x_N$ be points in $[0,\infty)$ with $0 \le x_1 <
x_2< \dots < x_N$. Now define $T$ from $\sH=L^{2}[0,\infty)$ to
$\sK$ by
\begin{equation}\label{finite}
 T=\left\{\,\left\{f, {\sum}^{N}_{n=1} f(x_n)e_{n} \right\} :\, f \in
 \cD\,\right\}.
\end{equation}
Then in a similar way it is seen that
\[
 T^*=\spn\{e_1,\ldots,e_N\}^\perp \times \{0\}, \quad
T^{**}=\sH \times\spn\{e_1,\ldots,e_N\}.
\]
Hence, again, $T$ is densely defined singular operator.

In particular, if $c \in [0,\infty)$ and
$T$ maps $f \in \cD\subset \sH=L^{2}[0,\infty)$ to $\sK=\dC$
by $f \to f(c)$, then
\[
 T^*=\{0\} \times \{0\}, \quad
T^{**}=\sH \times \dC,
\]
and the operator $T$ (linear functional) is singular.
\end{example}

\section{Orthogonal range decompositions}\label{sec3}

The main objects in this paper involve orthogonal range decompositions
of operators and relations. Here the definition and some properties
are given which will be repeatedly used in the rest of the paper.

\begin{definition}\label{osde}
Let $T$, $T_{1}$, and $T_{2}$ be relations from the Hilbert space $\sH$
to the Hilbert space $\sK$.
The sum $T=T_{1}+T_{2}$ is said to be
an \textit{orthogonal range decomposition} of $T$ if
\begin{equation}\label{qp0}
 \dom T=\dom T_1=\dom T_2 \quad \mbox{and}
 \quad \ran T_{1} \perp \ran T_{2}.
\end{equation}
\end{definition}

In the next lemma the orthogonal range decompositions of a relation
$T$ from $\sH$ to $\sK$ are characterized by orthogonal projectors
in $\sK$.

\begin{lemma}\label{osd}
Let $T$ be a  relation from the Hilbert space $\sH$ to the Hilbert
space $\sK$. Then the following statements hold:
\begin{enumerate}[ $(\rm i)$]
\item If $T$ has an orthogonal range decomposition $T=T_{1}+T_{2}$,
then there exists an orthogonal projection $Q$ in $\sK$ such that
\begin{equation}\label{qp12+}
Q\, \mul T \subset \mul T,
\end{equation}
and, in addition,
\begin{equation}\label{qp12}
 T_1=(I-Q)T \text{ and } T_2=QT.
\end{equation}

\item If $Q$ is an orthogonal projection in $\sK$ such that
\eqref{qp12+} is satisfied,  then $T$ has the orthogonal range
decomposition $T=T_{1}+T_{2}$ such that  \eqref{qp12} holds.
\end{enumerate}
\end{lemma}

\begin{proof}
(i) Let $Q$ be the orthogonal projection from $\sK$ onto $\cran
T_2$. Then clearly
\[
 QT_{2}=\{\,\{f,Qf'\}:\, \{f,f'\} \in T_{2}\,\} =T_{2}.
\]
Since $\ran T_1 \subset \ker Q$ the product $QT_1$ is the zero
operator on $\dom T_1$. It follows from the definition of the sum
that
\[
\begin{split}
 Q(T_1+T_2)&=\{\,\{f,Q(h+k)\} :\, \{f,h\} \in T_1, \,\, \{f,k\} \in T_2\,\}\\
 &=\{\,\{f,Qh+Qk)\} :\, \{f,h\} \in T_1, \,\, \{f,k\} \in T_2\,\}=QT_{1}+QT_{2}.
\end{split}
\]
Since $\dom T=\dom T_{1}=\dom T_{2}$ one sees that
$QT_{1}+QT_{2}=T_{2}$. Hence one obtains $QT=T_{2}$.
Now likewise one has $(I-Q)(T_1+T_2) =(I-Q)T_{1}+(I-Q)T_{2}$.
Note that $(I-Q)T_{2}=0$ while $(I-Q)T_{1}=T_{1}$.
Hence one obtains $(I-Q)T=T_{1}$. Thus \eqref{qp12} has been shown.
The identity \eqref{qp12+} follows from Lemma \ref{orth+}.

(ii) This statement is clear from Lemma \ref{orth+}.
 \end{proof}

The terminology of orthogonal range decomposition of $T=T_1+T_2$ in
Definition \ref{osde} refers to the fact  that $ \ran T_{1} \perp
\ran T_{2}$. Thus, if $T$ has the above orthogonal range
decomposition then for every $\{f,f'\} \in T$ one has
\[
 \{f,f'\}=\{f, f_{1}'+f_{2}'\}, \quad \mbox{where} \quad
 \{f,f_{1}'\}\in T_{1}, \,\,  \{f,f_{2}'\}\in T_{2}, \,\, f_{1}' \perp f_{2}'.
\]
Recall from \eqref{rm1} that in this case
\[
 \ran T \subset \ran T_1 \oplus \ran T_2.
\]
However, the following corollary shows that it is not necessarily a consequence
of an orthogonal range decomposition $T=T_1+T_2$
that equality holds in the above  inclusion.

\begin{corollary}\thlab{randecom}
Let $T$ be a  relation from the Hilbert space $\sH$ to the Hilbert
space $\sK$ and let $T=T_1+T_2$ be an orthogonal range decomposition of $T$.
Then the following conditions are equivalent:
\begin{enumerate}[ $(\rm i)$]
\item $\ran T_1 \subset \ran T$;

\item $\ran T_2 \subset \ran T$;

\item $\ran T=\ran T_1 \oplus \ran T_2$.
\end{enumerate}
Moreover, the following statements are equivalent:
\begin{enumerate}[ $(\rm i)$]
\setcounter{enumi}{4}

\item $\ran T_1 \subset \cran T$;

\item $\ran T_2 \subset \cran T$;

\item $\cran T = \cran T_1\oplus \cran T_2$.
\end{enumerate}
\end{corollary}

\begin{proof}
Assume that $Q$ is an orthogonal projection as in Lemma \ref{osd}.

(i) $\Leftrightarrow$ (ii) If $k \in \ran T_2$, then $k=Pf'$ for
some $\{f,f'\} \in T$ and
\[
f'=(I-P)f'+Pf' \in \ran T;
\]
see \eqref{qp12}. Hence, if (i) holds then $(I-P)f'\in \ran T$ and thus
also $k=Pf'\in\ran T$, i.e., (ii) follows. The reverse implication
is proved in the same way.

(i), (ii) $\Leftrightarrow$ (iii) Since automatically $\ran T
\subset \ran T_1 \oplus \ran T_2$, the inclusions in (i) and (ii)
imply the reverse inclusion and, thus, (iii) follows. The converse
statement is clear.

The equivalence of (iv), (v), and (vi) is seen with straightforward
modifications of the above arguments.
\end{proof}

To describe the components in an orthogonal range decomposition
of a relation the following lemma concerning projected relations will
play an important role.

\begin{lemma}\label{Lebtypelem}
Let $T$  be a linear relation from a Hilbert space
$\sH$ to a Hilbert space $\sK$ and let $Q$ be an orthogonal
projection in $\sK$.
 Then the following statements hold:
\begin{enumerate}[$(\rm i)$]
\item The relation $(I-Q)T$ is an operator if and only if
\begin{equation}\label{Mleb0}
   \mul T \subset \ran Q.
\end{equation}
\item The relation $(I-Q)T$ is regular if and only if
\begin{equation}\label{Mleb1}
   \clos(\ker Q \cap \dom T^* )=\ker Q,
\end{equation}
in which case
\begin{equation}\label{Mleb11}
\mul T^{**} \subset \ran Q.
\end{equation}
\item The relation $QT$ is singular if and only if
\begin{equation}\label{Mleb2}
    \ran Q \cap \dom T^* \subset \ker T^*,
\end{equation}
in which case
\begin{equation}\label{Mleb22}
( QT )^{**}=\cdom T^{**}\times (\ran Q \cap (\ran Q \cap \dom T^{*})^{\perp}).
\end{equation}
 \end{enumerate}
\end{lemma}

\begin{proof}
(i) Observe that $\mul (I-Q)T=(I-Q) \mul T=\{0\}$ (cf.
\eqref{prod0+}) if and only if $\mul T \subset \ker (I-Q)=\ran Q$.

(ii) Recall that   $(I-Q)T$ is regular
if and only if its adjoint $T^*(I-Q)$ is densely defined;
cf. Proposition \ref{stoch0}.
By Lemma \ref{PT} this is the case precisely when
\[
 \clos (\dom T^* \cap \ran (I-Q)) =\ran (I-Q),
\]
which is \eqref{Mleb1}. It follows from \eqref{Mleb1} that
\[
 \ker Q =\clos( \ker Q \cap \dom T^*) \subset \ker Q \cap \cdom T^* ,
\]
or $\ker Q \subset \cdom T^{*}$.
Taking orthogonal complements gives \eqref{Mleb11}.

(iii) Recall that $QT$ is singular if and only if
$\dom T^*Q \subset \ker T^*Q$;
cf. Proposition \ref{Tsinglemma0}.
Observe that Lemma \ref{PT} gives
\begin{equation}\label{PTsing0}
 \dom T^*Q
 =\ker Q \oplus (\dom T^* \cap \ran Q),
 \quad
  \ker T^*Q=\ker Q \oplus (\ker T^* \cap \ran Q).
\end{equation}
Hence $\dom T^*Q \subset \ker T^*Q$
if and only if $\dom T^* \cap \ran Q \subset \ker T^* \cap \ran Q$,
which is precisely \eqref{Mleb2}.
Now assume that $QT$ is singular, then also $(QT)^{*}=T^{*}Q$ is singular
and therefore
\begin{equation}\label{humhum}
 T^{*}Q=\dom T^{*}Q \times \mul T^{*}Q.
\end{equation}
Use Lemma \ref{PT} and \eqref{prod0++} to obtain
\[
\dom (T^{*}Q)=\ker Q  \oplus (\ran Q \cap \dom T^{*}), \quad
\mul T^{*}Q=  \mul T^{*}.
\]
Taking adjoints in \eqref{humhum} and applying
Proposition~\ref{Tsinglemma0} (i) leads to \eqref{Mleb22}.
 \end{proof}

The main emphasis of this paper will be on the following subclass
of the orthogonal range decompositions in Definition \ref{osde}.

\begin{definition}\label{osde+}
Let $T$, $T_{1}$, and $T_{2}$ be relations from the Hilbert space $\sH$
to the Hilbert space $\sK$.
The sum $T=T_{1}+T_{2}$ is said to be
a \textit{distinguished orthogonal range decomposition} of $T$ if
it is an orthogonal range decomposition
and if, in addition, $T_{1}$ is an operator.
\end{definition}

Note that if $T$ in Definition \ref{osde+} is an operator and
$T=T_{1}+T_{2}$ is a distinguished  orthogonal range decomposition
of $T$, then automatically $T_2$ is an operator. Furthermore, one
should be aware that also in a distinguished orthogonal range
decomposition  of $T$ it is not necessarily true that $T_1$ is
contained in $T$; cf. Corollary \ref{randecom}. The following
description of a distinguished orthogonal range decomposition is a
direct consequence of Lemma \ref{osd} and (i) in
Lemma~\ref{Lebtypelem}.

\begin{corollary}\label{osdd}
Let $T$ be a  relation from the Hilbert space $\sH$ to the Hilbert
space $\sK$. Then the following statements hold:
\begin{enumerate}[ $(\rm i)$]
\item If $T$ has a distinguished orthogonal range decomposition $T=T_{1}+T_{2}$,
then there exists an orthogonal projection $Q$ in $\sK$ such that
\begin{equation}\label{qp12++}
 \cmul T \subset \ran Q,
\end{equation}
and, in addition,
\begin{equation}\label{qp12+++}
 T_1=(I-Q)T \text{ and } T_2=QT.
\end{equation}

\item If $Q$ is an orthogonal projection in $\sK$ such that
\eqref{qp12++} is satisfied, then $T$ has the orthogonal range decomposition
$T=T_{1}+T_{2}$ such that  \eqref{qp12+++} holds.
\end{enumerate}
\end{corollary}

\begin{proof}
By Lemma~\ref{Lebtypelem}~(i) $(I-Q)T$ is an operator if and only if
$\mul T \subset \ran Q$. This last condition implies in particular
that $Q\, \mul T = \mul T$ and $\cmul T \subset \ran Q$. In this
case $T=(I-Q)T+QT$ is automatically an orthogonal range
decomposition. The remaining statements in (i) and (ii) are obtained
from Lemma~\ref{osd}.
 \end{proof}

The rest of this paper is devoted to classes of distinguished
orthogonal range decompositions of the form $T=T_{1}+T_{2}$ with the
additional requirement that the component $T_2$ is singular. First
the case where $T_{1}$ is a closable operator (i.e., $T_{1}$ is
regular) and $T_{2}$ is singular is studied. The weaker case where
$T_{1}$ is just an operator and $T_{2}$ is singular and its
connections with the earlier case is briefly treated after that. In
this analysis the descriptions in items (ii) and (iii) of Lemma
\ref{Lebtypelem} will be often used.

\section{Lebesgue decompositions for linear relations}\label{sec4}

Let $T$ be a relation from a Hilbert space $\sH$ to a Hilbert space
$\sK$. The following orthogonal  decomposition of the space
$\sK$,
\[
 \sK=\cdom T^* \oplus \mul T^{**},
\]
induces a corresponding  distinguished orthogonal range
decomposition of the relation $T$ itself; cf. \cite{HSeSnSz,HSnSz}.
This section gives a short self-contained treatment of this induced
decomposition of $T$ with some central properties that will be
relevant for the analysis in later sections.

Define the  \textit{regular part} $T_{\rm  reg}$ and the
\textit{singular part} $T_{\rm  sing}$ of $T$ by
\begin{equation}\label{resi}
 T_{\rm reg}=(I-P)T, \quad  T_{\rm sing}=PT,
\end{equation}
where $P$ is the orthogonal projection from $\sK$ onto $\mul
T^{**}$.
The regular and singular parts have the following
properties.

\begin{theorem}\label{jor0}
Let $T$  be a relation from a Hilbert space
$\sH$ to a Hilbert space $\sK$.
Then the relations $T_{\rm  reg}$ and $T_{\rm   sing}$ in \eqref{resi}
have the following properties:
\begin{enumerate}[ $ (\rm i)$]
\item $( T_{\rm  reg} )^{**}$ is an operator,
i.e., $T_{\rm reg}$ is regular;
\item $( T_{\rm  sing} )^{**}=\cdom T^{**}\times \mul T^{**}$,
i.e., $T_{\rm sing}$
is singular,
\end{enumerate}
and $T$ has the distinguished orthogonal range decomposition:
\begin{equation}\label{HSSSdec}
 T= T_{\rm reg} +  T_{\rm sing}.
\end{equation}
\end{theorem}

\begin{proof}
Apply Lemma \ref{Lebtypelem} with $Q=P$
so that $\ran Q=\mul T^{**}$
and $\ker Q=\cdom T^{*}$. Then
\[
   \clos(\ker Q \cap \dom T^* )=\cdom T^{*}=\ker Q,
\]
which implies that $(I-P)T$ is regular. Moreover,
\[
 \ran Q \cap \dom T^*= \mul T^{**} \cap \dom T^{*}=\{0\}  \subset \ker T^*,
\]
which implies that $PT$ is singular. The form of $(PT)^{**}$
is also clear from Lemma \ref{Lebtypelem}.

To show the identity \eqref{HSSSdec},
observe that $\mul T \subset \mul T^{**}$.
Since $P$ is the orthogonal projection onto $\mul T^{**}$
one sees that $P$ maps $\mul T$ into $\mul T$.
Therefore \eqref{HSSSdec} follows
from Lemma \ref{orth+}.
\end{proof}

\begin{definition}\label{DefLebdeco}
The distinguished orthogonal range decomposition in \eqref{HSSSdec} is
called the \textit{Lebesgue decomposition} of the relation $T$.
\end{definition}

The decomposition \eqref{HSSSdec} is the abstract variant of the
Lebesgue decomposition of a measure in terms of another measure; cf.
\cite{HSeSnLeb}. This decomposition for relations was established in
\cite[Theorem~4.1]{HSeSnSz}.  In the case that $T$ is an operator
such a decomposition has been first proved in \cite{J,Ota1}.

\begin{corollary}\label{CorRegSing}
Let $T$ be a relation  from a Hilbert space $\sH$ to a Hilbert space
$\sK$. Then the following equivalences hold:
\begin{equation}\label{regu}
 T \mbox{ is regular} \quad \Leftrightarrow \quad  T=T_{\rm   reg} \quad
 \Leftrightarrow \quad T_{\rm sing}=0,
\end{equation}
and
\begin{equation}\label{regu1}
 T \mbox{ is singular} \quad \Leftrightarrow \quad  T=T_{\rm sing} \quad
 \Leftrightarrow \quad T_{\rm reg}=0.
\end{equation}
\end{corollary}

\begin{proof}
The last equivalences in \eqref{regu} and \eqref{regu1}
are clear from \eqref{prod0+} and \eqref{HSSSdec}:
\[
T=(I-P)T \quad \Leftrightarrow \quad PT=0,
\quad \mbox{and} \quad
T=PT \quad \Leftrightarrow \quad (I-P)T=0.
\]

If $T$ is regular, i.e., $\mul T^{**}=\{0\}$,
then $P=0$ and $T=T_{\rm reg}$.
Conversely, if $T=T_{\rm reg}$ then
$T^{**}=(T_{\rm reg})^{**}$ is an operator,
i.e., $T$ is regular.

If $T$ is singular, i.e.,
$T^{**}=\dom T^{**} \times \mul T^{**}$, then
$(I-P)T \subset (I-P)T^{**}$, shows that
$T_{\rm reg}=0$ and thus $T=T_{\rm sing}$.
Conversely, if $T=T_{\rm sing}$ then
 it follows that $T^{**}=(T_{\rm sing})^{**}
=\cdom T^{**}\times \mul T^{**}$, i.e., $T$ is singular.
\end{proof}

By Theorem \ref{jor0} the regular part $T_{\rm reg}$ is closable. In
order to find its closure note the following property
\begin{equation}\label{chart}
 T_{\rm reg} \subset T^{**}.
\end{equation}
To see this observe for any $\{f,f'\} \in T$ that
\[
 \{f, (I-P)f'\}=\{f,f'\} -\{0,Pf'\}, \quad \{f,f'\} \in T \subset T^{**},
 \quad
 \{0, Pf'\} \in T^{**},
\]
which gives \eqref{chart}. In particular it follows from \eqref{chart}
that $(T_{\rm reg})^{**} \subset T^{**}$.
Now consider the Lebesgue decomposition
for the closure $T^{**}$ of $T$.
The regular and singular parts of   $T^{**}$
are given by
\[
(T^{**})_{\rm  reg}=(I-P)T^{**}, \quad (T^{**})_{\rm  sing}=PT^{**},
\]
since the corresponding
orthogonal projection is given by the same projection $P$.
Thus $(T^{**})_{\rm reg}=(I-P)T^{**}$ is a regular operator,
while $(T^{**})_{\rm  sing}=PT^{**}$
is a singular relation.

\begin{theorem}\label{jor}
Let $T$ be a relation  from a Hilbert space $\sH$ to a Hilbert space
$\sK$.
The relation $T^{**}$ admits the Lebesgue decomposition:
\begin{equation}\label{HSSSdec+}
 T^{**}= (T^{**})_{\rm   reg} +  (T^{**})_{\rm  sing},
\end{equation}
where the regular part $(T^{**})_{\rm   reg}$ satisfies
\begin{equation}\label{HSSSdec+l}
  (T^{**})_{\rm reg} \subset T^{**},
\end{equation}
and $(T^{**})_{\rm reg}$ is closed.
The regular and singular parts of $T^{**}$ satisfy
\begin{equation}\label{HSSSdec+0}
({T^{**}})_{\rm   reg} =({T}_{\rm   reg})^{**}, \quad
 ({T^{**}})_{\rm   sing}=\dom T^{**} \times \mul T^{**}.
\end{equation}
The singular part $(T^{**})_{\rm  sing}$ is closed if and only $\dom
T^{**}$ is closed. Moreover,
\begin{equation}\label{nelo}
(({T^{**}})_{\rm  sing})^{**} =\cdom T^{**} \times \mul T^{**}
=  ({T}_{\rm  sing})^{**}.
\end{equation}
\end{theorem}

\begin{proof}
The Lebesgue decomposition \eqref{HSSSdec+} follows directly from
Theorem \ref{jor0}, since the orthogonal projection $P$ maps onto
$\mul T^{**}$.
 Furthermore, the inclusion \eqref{HSSSdec+l} follows from
\eqref{chart}. In order to show that $(T^{**})_{\rm  reg}$ is
closed, let $\{f_{n},g_{n}\} \in (T^{**})_{\rm  reg}$ converge to
$\{f,g\} \in \sH\times \sK$. By \eqref{HSSSdec+l}  one has that
$\{f,g\} \in T^{**}$. Since $\{f_{n},g_{n}\} \in (T^{**})_{\rm reg}$
it follows that $g_{n} \in \cdom T^{*}$ and hence $g \in \cdom
T^{*}$, i.e., $g=(I-P)g$. This implies that $\{f,g\} \in
(I-P)T^{**}=(T^{**})_{\rm reg}$.

To see the first identity in  \eqref{HSSSdec+0} observe that
\[
 ((I-P)T)^*=T^*(I-P)=((I-P)T^{**})^*.
\]
One concludes that $(T_{\rm  reg})^*=((T^{**})_{\rm reg})^*$
and,  consequently,
\[
(T_{\rm  reg})^{**}=((T^{**})_{\rm  reg})^{**}.
\]
Since $(T^{**})_{\rm  reg}$ is closed,
the first identity in  \eqref{HSSSdec+0} follows.

To see the second identity in  \eqref{HSSSdec+0}
observe that it follows
directly from the definition that
\[
(T^{**})_{\rm sing}=PT^{**} \subset \dom T^{**} \times \mul T^{**}.
\]
To see the reverse inclusion let $\{f,g\} \in \dom T^{**} \times
\mul T^{**}$. Then there exists $f'$ such that $\{f,f'\} \in T^{**}$
and it is clear from \eqref{HSSSdec+l} that $\{f, (I-P) f'\} \in
T^{**}$, so that
\[
 \{f, (I-P)f'+g\} \in T^{**}.
\]
By definition $\{f,g\} =\{f, P((I-P)f'+g)\} \in (T^{**})_{\rm sing}$.
Hence the second identity in  \eqref{HSSSdec+0}
has been established.

The second identity in \eqref{HSSSdec+0} shows that $(T^{**})_{\rm
sing}$ is closed if and only $\dom T^{**}$ is closed. Moreover, by
taking closures in this second identity and comparing the result
with the formula (ii) in Theorem \ref{jor0} the last assertion in
\eqref{nelo} follows.
\end{proof}

Let $T$ be a relation  from a Hilbert space $\sH$ to a Hilbert space
$\sK$. The closed relation $T^{**}$ can also be written as an
orthogonal (componentwise) sum
\[
 T^{**}= (T^{**})_{\rm reg} \hoplus (\{0\} \times \mul T^{**}).
\]
The regular part $(T^{**})_{\rm reg}$ now serves as an orthogonal
operator part of $T^{**}$; cf. \cite{HSnSz}. The identity in the
Lebesgue decomposition \eqref{HSSSdec} persists under closures.

\begin{corollary}
The relation $T^{**}$ admits the following   decomposition:
\begin{equation}\label{HSSSdec3--}
 T^{**}=(T_{\rm  reg})^{**} + (T_{\rm  sing})^{**}.
\end{equation}
\end{corollary}

\begin{proof}
Recall from \eqref{HSSSdec+} that $T^{**}$ admits the Lebesgue
decomposition
\[
 T^{**}=(T^{**})_{\rm reg}+(T^{**})_{\rm sing}.
\]
Moreover, by Theorem \ref{jor} one has
\[
(T^{**})_{\rm reg}=(T_{\rm reg})^{**}, \quad
(T^{**})_{\rm sing}=\dom T^{**} \times \mul T^{**} \subset (T_{\rm sing})^{**},
\]
and thus the left-hand side of \eqref{HSSSdec3--} is included in the
right-hand side. Here equality prevails, since $\dom (T_{\rm
reg})^{**}=\dom T^{**}$ and $\dom (T_{\rm reg})^{**}\cap \dom
(T_{\rm sing})^{**}=\dom T^{**}$; cf. Theorem \ref{jor0}~(ii) and
\eqref{HSSSdec+0}.
\end{proof}

\section{Lebesgue type decompositions for relations}\label{sec5}

In this section the notion of Lebesgue decomposition appearing in
Definition \ref{DefLebdeco} is formalized by claiming from its
additive components appropriate properties which imitate closely the
ones familiar from  measure theory. This leads to a general
definition of Lebesgue type decomposition for linear relations. The
main purpose here is to describe all such decompositions of a given
linear relation or operator.

\begin{definition}\label{defLebtype}
Let $T$  be a linear relation from a Hilbert space $\sH$ to a
Hilbert space $\sK$. Then $T$ is said to have a \textit{Lebesgue
type decomposition}
if it admits a distinguished
orthogonal range decomposition $T=T_1+T_2$
where $T_1$ is regular and $T_2$ is singular.
\end{definition}

Note that the Lebesgue decomposition \eqref{HSSSdec} in Theorem
\ref{jor0} is an example of a Lebesgue type decomposition. The
Lebesgue type decompositions of  a relation $T$ are characterized in
the following theorem.

\begin{theorem}\label{Lebtype}
Let $T$  be a linear relation from a Hilbert space $\sH$ to a
Hilbert space $\sK$ and let $\sM$ be a closed linear subspace in
$\sK$ such that
\begin{equation}\label{Mleb}
   \clos(\sM^\perp \cap \dom T^* )=\sM^\perp, \quad
   \sM \cap \dom T^* \subset \ker T^*.
\end{equation}
Then the relations $(I-P_{\sM})T$ and $P_{\sM}T$ satisfy
\begin{enumerate}[ $ (\rm i)$]
\item $((I-P_{\sM})T)^{**}$ is an operator,
i.e., $(I-P_{\sM})T$ is regular;
\item $( P_{\sM}T )^{**}
=\cdom T^{**}\times (\sM \cap (\sM \cap \dom T^{*})^{\perp})$,
i.e., $P_{\sM}T$ is singular,
\end{enumerate}
and $T$ has a Lebesgue type decomposition of the form
\begin{equation}\label{lebdeco}
T = (I-P_{\sM}) T + P_{\sM}T.
\end{equation}
Conversely, if $T$ has a Lebesgue type decomposition $T=T_1+T_2$ as
in Definition~\ref{defLebtype}, then the subspace $\sM=\cran T_2$
satisfies \eqref{Mleb} and generates this decomposition via
\eqref{lebdeco}: $T_1=(I-P_{\sM}) T$ and $T_2=P_{\sM}T$.
\end{theorem}

\begin{proof}
Assume that $\sM \subset \sK$ is a closed linear subspace which
satisfies the conditions \eqref{Mleb}. Then the assertions (i) and
(ii) follow from Lemma \ref{Lebtypelem} with $Q=P_{\sM}$. Moreover,
by Lemma \ref{Lebtypelem} $\mul T^{**} \subset \sM$. In particular,
this implies that $\mul T \subset \sM$, so that $P_{\sM} \,\mul T
\subset \mul T$. Hence Lemma \ref{orth+} gives the identity
\eqref{lebdeco}.

To prove the converse let $T=T_1+T_2$ be a Lebesgue type
decomposition of $T$ and let $\sM=\cran T_2$. Then (the proof of)
Lemma~\ref{osd} shows that $T_1=(I-P_{\sM})T$ and $T_2=P_{\sM}T$.
That $\sM=\cran T_2$ satisfies the properties in \eqref{Mleb}
follows now from Lemma \ref{Lebtypelem}.
\end{proof}

The first condition in \eqref{Mleb} in conjunction
with Lemma \ref{Lebtypelem} leads to the following
alternative description.

\begin{lemma}\label{corl}
Let $T$  be a linear relation from a Hilbert space
$\sH$ to a Hilbert space $\sK$ and let
$\sM$ be a closed linear subspace in $\sK$.
If \eqref{Mleb} is satisfied then there exists a closed
linear subspace $\sL \subset \cdom T^{*}$, such that
\begin{equation}\label{L}
 \sM=\mul T^{**} \,\oplus \sL,
\end{equation}
and, in terms of $\sL$, the conditions in \eqref{Mleb}
 are equivalent to
\begin{equation}\label{Lleb}
\clos( \sL^\perp \cap \dom T^* ) =\sL^\perp \cap \cdom T^*,
\quad
\sL \cap \dom T^* \subset \ker T^*,
\end{equation}
respectively. Furthermore $( P_{\sM}T )^{**}$ has the form
\begin{equation}\label{Llebsing1}
 ( P_{\sM}T )^{**}
=\cdom T^{**}\times ((\mul T^{**} \oplus \sL) \cap (\sL \cap \dom T^{*})^{\perp}).
\end{equation}
Conversely, if $\sL$ is a closed linear subspace in  $\cdom T^{*}$,
such that \eqref{Lleb} is satisfied, then $\sM$ defined by \eqref{L}
is a closed linear subspace of $\sK$ which satisfies \eqref{Mleb}.
\end{lemma}

\begin{proof}
Assume that \eqref{Mleb} is satisfied, then $\mul T^{**} \subset
\sM$ by Lemma \ref{Lebtypelem}. Hence there exists a closed linear
subspace $\sL \subset \cdom T^{*}$, such that \eqref{L} holds. Note
that then one has
\[
 \sM^\perp=\cdom T^* \cap \sL^\perp, \quad
 \dom T^* \cap \sM^\perp=\dom T^* \cap \sL^\perp.
\]
Therefore, the first identity in \eqref{Mleb} is satisfied if and
only if the closed linear subspace $\sL \subset \cdom T^{*}$
satisfies
\begin{equation}\label{uks}
\clos( \sL^\perp \cap \dom T^*) = \sL^\perp \cap \cdom T^*.
\end{equation}
Notice also that with \eqref{L} one has
\[
\sM \cap \dom T^{*} =(\mul T^{**} \oplus \sL) \cap \dom T^{*}
=\sL \cap \dom T^{*},
\]
so that the second identity in \eqref{Mleb} is equivalent to
\begin{equation}\label{kaks}
\sL \cap \dom T^{*} \subset \ker T^{*}.
\end{equation}
The formula \eqref{Llebsing1} is clear from item (ii) in Theorem
\ref{Lebtype}.
\end{proof}

The notation $\sM=\mul T^{**} \oplus \sL$ where $\sL$ is a closed
linear subspace of $\cdom T^{*}$ will be used throughout this paper.
Here the choice $\sL=\{0\}$ produces the Lebesgue decomposition of
$T$ in Theorem \ref{jor0}. In general, the choice of the subspace
$\sM$ or $\sL$ in Theorem \ref{Lebtype} and Lemma \ref{corl} to
derive a Lebesgue type decomposition $T=T_1+T_2$ via \eqref{lebdeco}
is not unique. For instance, for any closed linear subspace $\sL
\subset \ker T^{*}$ one can easily check that
\begin{equation}\label{lt1}
 (I-P_{\sM})T=T_{\rm reg}, \quad P_{\sM}\,T=T_{\rm sing},
\end{equation}
i.e. all closed linear subspaces $\sL  \subset \ker T^{*}$ produce the Lebesgue
decomposition of $T$. There is a further restriction in the
conditions \eqref{Mleb} or, equivalently, in the conditions \eqref{Lleb},
which makes the choice of $\sM$ and $\sL$ unique for
each Lebesgue type decomposition and leads to a one-to-one
correspondence between all Lebesgue type decomposition of $T$ and
the subspaces $\sM=\mul T^{**} \oplus \sL$.

\begin{theorem}\label{Lebtype2}
Let $T$ be a relation from the Hilbert space $\sH$ to the Hilbert
space $\sK$. Then the Lebesgue type decompositions $T=T_1+T_2$ of
$T$ are in one-to-one correspondence via $\sM=\cran T_2=\mul T^{**}
\oplus \sL$ with the closed linear subspaces $\sL \subset \cdom T^*
\setminus \dom T^{*}$ which satisfy the condition
\begin{equation}\label{Lleb2}
\clos(\sL^\perp \cap \dom T^*) =\sL^\perp \cap \cdom T^*.
\end{equation}
In particular, in this case
\begin{equation}\label{Lleb2+}
(T_2)^{**} =\cdom T^{**}\times (\mul T^{**} \oplus \sL).
\end{equation}
\end{theorem}

\begin{proof}
Assume that the relation $T$ from $\sH$ to $\sK$ has a Lebesgue type
decomposition of the form
\[
 T=T_1+T_2,
\]
where $T_{1}$ and $T_{2}$ are relations from $\sH$ to $\sK$
such that
\[
 \dom T_{1}=\dom T_{2}=\dom T,
 \quad \ran T_{1} \perp \ran T_{2},
\]
and $T_{1}$ is regular while $T_{2}$ is singular. Define $\sM=\cran
T_2$, so that $\sM$ is a closed linear subspace of $\sK$, and let
$P_\sM$  be the orthogonal projection onto $\cran T_2$. Then by
Lemma \ref{osd}
\[
T_1=(I-P_{\sM})T, \quad T_2=P_{\sM}T.
\]
Moreover, Lemma \ref{corl} shows that $\sM$ is of the form $\sM=\mul
T^{**} \oplus \sL$, where $\sL$ is a closed linear subspace of
$\cdom T^{*}$ satisfying \eqref{Lleb}.

It will be shown that with the choice $\sM=\cran T_2$ the last
condition in \eqref{Lleb}, namely $\sL \cap \dom T^{*} \subset \ker
T^{*}$, actually reduces to $\sL \cap \dom T^{*}=\{0\}$. For this
purpose assume that $g \in \sL \cap \dom T^{*}$. Then $g \in \ker
T^{*}$ and $g \in \sL \subset \sM$. Hence $g=P_{\sM}g$ is orthogonal
to $\ran T$ and it follows that $g$ is orthogonal to $P_{\sM} \,\ran
T=\ran P_{\sM}T=\ran T_{2}$. Since $\sM=\cran T_{2}$, it follows
that $g=0$. Thus $\sL$ satisfies $\sL \subset \cdom
T^{**}\setminus\dom T^{**}$ and \eqref{Lleb2}.

It is clear from Theorem \ref{Lebtype} and Lemma \ref{corl} that the
converse holds. In other words, every linear subspace $\sL \subset
\cdom T^{*} \setminus \dom T^{*}$ which satisfies \eqref{Lleb2}
gives rise to a Lebesgue type decomposition of $T$.

The formula \eqref{Lleb2+} is immediate from \eqref{Llebsing1} in
Lemma \ref{corl}. Since
\[
 \sM=\cran T_2=\mul (T_2)^{**},
\]
it follows from \eqref{Lleb2+} that the choice of
$\sL \subset \cdom T^{*} \setminus \dom T^{*}$ uniquely determines
$\sM(=\mul (T_2)^{**})$ and $T_2=P_{\sM}T$ and, conversely, the choice of $T_2$
uniquely determines $\mul (T_2)^{**}$ and $\sL$. Hence, the
correspondence between $T_2$ and $\sL$ is one-to-one.
\end{proof}

Notice that the domain and range of a closed relation $T$ are
operator ranges; as shown in \eqref{ToperRange} this remains true
even for $T$ which is itself an operator range. In the case that
$\dom T^*$ is not closed one necessarily has $\dim (\cdom T^*/\dom
T^*)=\infty$; see e.g. \cite[Theorem~2.3]{FW}. In this case there
are a lot of closed subspaces $\sL\subset \cdom T^*\setminus\dom
T^*$. If, for instance, $e \in \cdom T^* \setminus \dom T^{*}$ then
also $e+f \in \cdom T^* \setminus \dom T^{*}$ for any $f\in \dom
T^*$. Moreover, each of the $1$-dimensional subspaces
$\sL=\spn\{e+f\}$ satisfies also the condition \eqref{Lleb2} in
Theorem \ref{Lebtype2}.

More generally the following result holds; for the proof just apply
Lemma \ref{nneeww} with $\sM=\cdom T^{*}$ closed and $\sR=\dom
T^{*}$, which is an operator range:

\begin{lemma}\label{SA}
Let $T$ be a linear relation from the Hilbert space $\sH$ to the
Hilbert space $\sK$ and assume that $\dom T^*$ is not closed. Let
$\sL\subset \cdom T^* \setminus \dom T^{*}$ be a closed subspace and
let $P_\sL$ be the orthogonal projection from $\sH$ onto $\sL$. Then
the condition
\begin{equation}\label{PdomT*}
 P_\sL(\dom T^*)=\sL
\end{equation}
implies the condition \eqref{Lleb2} in Theorem \ref{Lebtype2}. Thus,
any closed subspace $\sL\subset \cdom T^* \setminus \dom T^{*}$
satisfying \eqref{PdomT*} gives rise to a Lebesgue type
decomposition which is different from the Lebesgue decomposition of
$T$.

In particular, the condition \eqref{PdomT*} holds if $\dim
\sL<\infty$ and thus there are infinitely many different Lebesgue
type decompositions for $T$ whenever $\dom T^*$ is not closed.
\end{lemma}

Assume in Theorem \ref{Lebtype2} that $T$ is a closable linear
operator. Then by Corollary \ref{CorRegSing} one sees that $T=T_{\rm
reg}$ and $T_{\rm sing}=0$. Hence in this case the Lebesgue
decomposition of $T$ is trivial.  The non-trivial Lebesgue type
decompositions $T=T_1+T_2$ of $T$ are in one-to-one correspondence
via $\sM=\cran T_2=\sL$ with the non-trivial closed linear subspaces
$\sL \subset \cdom T^* \setminus \dom T^{*}$ which satisfy the
condition \eqref{Lleb2}; in this case $\mul (T_2)^{**} =\cdom
T^{**}\times \sL$ . Hence as long as $\dom T^*$ is not closed,
\textit{any} closable operator $T$ has infinitely many non-trivial
Lebesgue type
decompositions, while its Lebesgue decomposition is completely trivial. \\

\section{The uniqueness of Lebesgue type decompositions}\label{sec6}

Let $T$ be a linear relation from the Hilbert space $\sH$ to the
Hilbert space $\sK$. It is a consequence of Lemma \ref{SA} that when
$\dom T^{*}$ is not closed there is no uniqueness: there exist
Lebesgue type decompositions of $T$ which are different from the
Lebesgue decomposition. In the case where  $\dom T^{*}$ is closed
there is a completely different behavior.

\begin{theorem}\label{uninew}
Let $T$  be a linear relation from a Hilbert space
$\sH$ to a Hilbert space $\sK$.
Then the following statements are equivalent:
\begin{enumerate}[ $(\rm i)$]
\item
the Lebesgue decomposition is the only Lebesgue
type decomposition of $T$;
\item $\dom T^*$ is closed;
\item $T_{\rm reg}$ is bounded.
\end{enumerate}
\end{theorem}

\begin{proof}
(i) $\Rightarrow$ (ii) If $\dom T^*$ is not closed, it follows from
Lemma \ref{SA} that there are nonzero finite-dimensional subspaces
$\sL\subset \cdom T^*\setminus\dom T^*$ inducing Lebesgue type
decompositions for $T$ which are different from its Lebesgue decomposition.

(ii) $\Rightarrow$ (i) Let $T=(I-P_{\sM})T+P_{\sM}T$ be a Lebesgue
type decomposition of $T$ generated by some closed linear subspace
$\sM\subset \sK$. According to Theorem \ref{Lebtype2} one can take
$\sM=\mul T^{**}\oplus \sL$, where $\sL\subset \cdom T^*\setminus
\dom T^*$. However, if $\dom T^*$ is closed this means that
$\sL=\{0\}$. Therefore the Lebesgue type decomposition of $T$
coincides with the Lebesgue decomposition of $T$.

(ii) $\Leftrightarrow$ (iii) It follows from Theorem \ref{jor} that
\begin{equation}\label{Tregbounded1}
 \dom T^{**}=\dom (T^{**})_{\rm reg}=\dom (T_{\rm reg})^{**}.
\end{equation}
This shows that $T_{\rm reg}$ as a closable operator is bounded
precisely when $\dom T^{**}$ is closed or, equivalently, when $\dom T^*$
is closed.
\end{proof}

\begin{corollary}\label{grijsz}
Let $T$ be a closable linear operator from the Hilbert space $\sH$
to the Hilbert space $\sK$. Then each of the statements (i)--(iii)
in Theorem \ref{uninew} is equivalent to $T$ being bounded.
\end{corollary}

Notice that the parametrization of the Lebesgue type decompositions
in Theorem \ref{Lebtype2} is being used in order to show the
uniqueness result in Theorem \ref{uninew}. Each of the items (i) and
(iii) is equivalent to the condition that $\dom T^{*}$  is closed in
item (ii). In the context of pairs of nonnegative bounded operators
a uniqueness result for the corresponding Lebesgue type
decompositions was obtained by Ando \cite{An}. Ando's proof was
based on arguments to derive the equivalence of (i) and (iii)
directly without the analysis via Lebesque type decompositions.

\begin{theorem}\label{ExtUnirep}
Let $T$  be a linear relation from a Hilbert space $\sH$ to a
Hilbert space $\sK$. The Lebesgue decomposition of $T$ is the only
Lebesgue type decomposition $T=T_1+T_2$ of $T$ whose regular part
$T_1$ has the property
\begin{equation}\label{char}
 T_1 \subset T^{**}.
\end{equation}
Furthermore, if the regular part $T_1$ of a Lebesgue type
decomposition $T=T_1+T_2$ has the property
\begin{equation}\label{charchar}
T_1 \quad \mbox{ is bounded,}
\end{equation}
then it coincides with the Lebesgue decomposition of $T$. In
particular, if $T_1\neq T_{\rm reg}$, then $T_1$ is an unbounded
closable operator.
\end{theorem}

\begin{proof}
Each Lebesgue type decomposition of $T$ is given by \eqref{lebdeco}
with the closed linear subspace $\sM=\cran T_2^{**}=\mul
T^{**}\oplus \sL$ as in Theorem \ref{Lebtype2}.

Now assume that \eqref{char} holds. Then $\{f,(I-P_{\sM})f'\}\in
T_1\subset T^{**}$ for every $\{f,f'\} \in T$ and hence it follows
from
\[
 \{f,f'\}=\{f,(I-P_{\sM})f'\}+\{0,P_{\sM}f'\} \in T
\]
that $\{0,P_{\sM}f'\} \in T^{**}$. Thus, $\ran T_2=P_{\sM}\ran
T\subset \mul T^{**}$ and, since $T_2$ is singular, this leads to
$\mul T_2^{**}=\cran T_2\subset \mul T^{**}$. Hence, $\sL=\{0\}$ and
by Theorem \ref{Lebtype2} the decomposition $T=T_1+T_2$ corresponds
to the Lebesgue decomposition of $T$.

To prove the second assertion assume that the regular part
$T_1=(I-P_{\sM})T$ is a bounded operator. Then
$T^{**}=T_1^{**}+T_2^{**}$ and hence $\mul T^{**}=\mul T_2^{**}$. By
Theorem \ref{Lebtype2} this means that $\sL=\{0\}$ which corresponds
to the Lebesgue decomposition of $T$. In particular, if $T_1$ is
bounded then $T_1=T_{\rm reg}$, which gives the last statement.
\end{proof}

Observe, that the condition \eqref{charchar} is equivalent to the
conditions (i)--(iii) in Theorem \ref{uninew} and, therefore, it
contains a further extension of the uniqueness result of Ando. A
simple application of Theorem \ref{uninew} gives also the following
uniqueness result.

\begin{corollary}\label{domtster}
Let $T$ be a linear relation from a Hilbert space $\sH$
to a Hilbert space $\sK$. Then $\dom T^{**} = \sH$ if and only
if $T^{*}$ is a bounded operator. If either condition holds,
then $T$ is densely defined and there is the following alternative:
\begin{enumerate}[ $(\rm i)$]
\item  $\cdom T^{*}= \sH$ and $T^{**} \in \bB(\sH,\sK)$;

\item  $\cdom T^{*} \neq \sH$ and $(I-P)T^{**} \in \bB(\sH,\sK)$.
\end{enumerate}
The Lebesgue decomposition $T=(I-P)T+PT$ is
the unique Lebesgue type decomposition of $T$.
Moreover, $(T_{\rm sing})^{**}=\sH \times (\dom T^*)^\perp$.
\end{corollary}

\begin{proof}
If $\dom T^{**}=\sH$, then
$\dom T^{*}$ is closed and $\mul T^{*}=\{0\}$.
By the closed graph theorem, $T^{*}$ is a bounded operator.
Conversely, if $T^{*}$ is a bounded operator,
then $\dom T^{*}$ is closed
and, hence, $\dom T^{**}$ is closed.
Since $T^{*}$ is an operator, it follows that $\dom T^{**}$ is dense.
Hence $\dom T^{**}=\sH$.

If either condition is satisfied,  then
$\cdom T=\cdom T^{**}$ shows that $T$ is densely defined.
In this case there are two possibilities:
\begin{enumerate}[ $(\rm i)$]
\item  $\cdom T^{*}= \sH$, in which case  $\mul T^{**}=\left\{0\right\}$
and $T^{**}$ is a closed operator.
By the closed graph theorem $T^{**} \in \bB(\sH, \sK)$.

\item $\cdom T^{*} \neq \sH$, in which case
$\mul T^{**} \neq \left\{0\right\}$.
Then the operator $(I-P)T^{**}$ is closable and everywhere
defined.  This implies that $(I-P)T^{**}$ is closed and, hence, bounded,
so that  $(I-P)T^{**} \in \bB(\sH,\sK)$.
\end{enumerate}
Moreover,  it follows from $(I-P)T \subset (I-P)T^{**}$ that the
regular part of $T$ is bounded. Hence the Lebesgue decomposition
$T=(I-P)T+PT$ is the unique Lebesgue type decomposition of $T$. The
last statement is clear from Theorem \ref{jor0}.
\end{proof}

Let $T$ be a linear relation from a Hilbert space $\sH$ to a Hilbert
space $\sK$ with $\dom T=\sH$. Then the inclusion $T \subset
\overline{T} =T^{**}$ shows that $\sH=\dom T \subset \dom T^{**}$,
so that $\dom T^{**} = \sH$. Hence, Corollary \ref{domtster} may be applied.
If $\cdom T^{*}=\sH$, then  $T \in \bB(\sH,\sK)$, while if $\cdom
T^{*} \neq \sH$, then $(I-P)T \in \bB(\sH,\sK)$. Note that if $T$ is
a linear operator from a Hilbert space $\sH$ to a Hilbert space
$\sK$ with $\dom T=\sH$ and $\cdom T^* \neq \sK$, then $T$ is not
closable and therefore not bounded.  An example of such an operator
can be found in \cite[p. 62]{AG81}.

\section{Weak Lebesgue type decompositions for linear relations}\label{sec7}

Let $T$ be a linear relation from a Hilbert space $\sH$ to a Hilbert
space $\sK$. Parallel to the earlier results one may ask for a weak
version of  the  Lebesgue type decomposition $T=T_{1}+T_{2}$, where
now $T_{1}$ is required to be an operator, which is not necessarily
closable, and $T_{2}$ is, as before, a singular relation. Such
decompositions are briefly treated in this section.

In order to obtain a weak version of the Lebesgue decomposition,
consider the following orthogonal decomposition of the space $\sK$:
\[
 \sK=(\cmul T)^{\perp} \oplus \cmul T.
\]
This decomposition implies a corresponding distinguished orthogonal
range decomposition of the relation $T$ itself.
Define the relations $T_{\rm op}$ and
$T_{\rm mul}$ from $\sH$ to $\sK$  by
\begin{equation}\label{Tm}
T_{\rm op}= (I-P_m)T, \quad T_{\rm mul}=P_mT,
\end{equation}
where $P_m$ is the orthogonal projection from $\sK$ onto $\cmul T$.
The next result is a weak version of Theorem \ref{jor0}.

\begin{theorem}\label{Tmlemma}
Let $T$  be a relation from a Hilbert space
$\sH$ to a Hilbert space $\sK$.
Then the relations $T_{\rm op}$ and $T_{\rm mul}$ in \eqref{Tm}
have the following properties:
\begin{enumerate}[ $ (\rm i)$]
\item $T_{\rm  op}$ is an operator,
\item $(T_{\rm mul})^{**}=\cdom T^{**} \times \cmul T$,
i.e., $T_{\rm mul}$ is singular,
\end{enumerate}
and $T$ has the distinguished orthogonal range decomposition:
\begin{equation}\label{decoms0}
 T=T_{\rm op}+T_{\rm mul}.
\end{equation}
\end{theorem}

\begin{proof}
Apply Lemma \ref{Lebtypelem} with $Q=P_{m}$. Since $\mul T \subset
\cmul T=\ran Q$, the component $T_{\rm op}$ is an operator.
Furthermore,
\[
 \ran Q \cap \dom T^{*}=\cmul T \cap \dom T^{*} \subset
 \mul T^{**}\cap \cdom T^{*}=\{0\},
\]
so that the relation $T_{\rm mul}$ is singular.
Due to $\ran Q \cap \dom T^{*}=\{0\}$
the identity in (ii)   follows from \eqref{Mleb22}.
Since $\mul T \subset \cmul T$ one has
\[
 P_{m} \, (\mul T) \subset \mul T.
\]
Hence, according to \thref{orth+}, the relation $T$ has the
orthogonal range decomposition \eqref{decoms0}.
\end{proof}

\begin{definition}
The distinguished orthogonal range decomposition in \eqref{decoms0} is called
the \textit{weak Lebesgue decomposition} of the relation $T$.
\end{definition}

Let $P$ be the orthogonal projection onto $\mul T^{**}$; then by
\eqref{mul} it follows that $\ran P_m\subset \ran P$ and
\[
(I-P)T_{\rm op}=(I-P)(I-P_m)T=(I-P)T.
\]
Hence the operator part $T_{\rm op}$ and the regular part
$T_{\rm reg}$ in \eqref{resi} are connected by
\begin{equation}\label{regg}
T_{\rm reg}=(I-P)T_{\rm op}.
\end{equation}

\begin{corollary}
The operator $T_{\rm op}$ is closable if and only if
$\cmul T=\mul T^{**}$.
In fact, the multivalued parts of the closures of $T$ and
$T_{\rm op}$ are connected by
\begin{equation}\label{Tm2+}
\mul T^{**}= \mul (T_{\rm op})^{**} \oplus \cmul T.
\end{equation}
\end{corollary}

\begin{proof}
In order to consider the Lebesgue decomposition of $T_{\rm op}$
note that by Lemma \ref{PT} one has
\[
\dom (T_{\rm op})^*=\ker (I-Q_{m}) \oplus (\dom T^*\cap \ran
(I-Q_m)).
\]
Since $\dom T^*\subset \ran (I-P)\subset \ran (I-Q_m)$, one obtains
\begin{equation}\label{domTm*}
\dom (T_{\rm op})^*=\ran Q_m \oplus \dom T^*=\cmul T  \oplus \dom
T^*.
 \end{equation}
 Taking orthogonal complements in \eqref{domTm*}   gives
\begin{equation}\label{Tm2++}
\mul (T_{\rm op})^{**}= \mul T^{**} \ominus \cmul T,
\end{equation}
from which \eqref{Tm2+} follows.
\end{proof}

\begin{corollary}\label{TmCor}
 The Lebesgue decomposition of $T_{\rm op}$
is given by
\begin{equation}\label{Tm2--}
 T_{\rm op}=(T_{\rm op})_{\rm reg}+(T_{\rm op})_{\rm sing},
\end{equation}
where the singular part $(T_{\rm op})_{\rm sing}$ is given by
\begin{equation}\label{Tm2-}
((T_{\rm op})_{\rm sing})^{**}=\cdom T^{**} \times \mul (T_{\rm
op})^{**}.
\end{equation}
The regular and singular parts of $T$ and $T_{\rm op}$ are connected
via
\begin{equation}\label{Tm2}
 T_{\rm reg}=(T_{\rm op})_{\rm reg}, \quad T_{\rm sing}=(T_{\rm op})_{\rm sing}+T_{\rm mul}.
\end{equation}
In particular, the Lebesgue type decompositions of $T$ and $T_{\rm
op}$ are unique simultaneously.
\end{corollary}

\begin{proof}
Let $P_r$ be the orthogonal projection onto $\mul (T_{\rm op})^{**}$
in $\sK$. Then $(T_{\rm op})_{\rm reg}=(I-P_{r}) T_{\rm op}$ and
$(T_{\rm op})_{\rm sing}=P_{r} T_{\rm op}$. By Theorem \ref{jor0}
the decomposition \eqref{Tm2--} holds and
\[
 ((T_{\rm op})_{\rm sing})^{**}=\cdom (T_{\rm op})^{**} \times \mul
 (T_{\rm op})^{**},
\]
which combined with
\[
 \cdom (T_{\rm op})^{**}=\cdom (T_{\rm op})= \cdom T=\cdom T^{**}
\]
gives \eqref{Tm2-}. The orthogonal decomposition  \eqref{Tm2+} shows
that $P=P_r+Q_m$ with $P_rQ_m=0$. Consequently, one obtains
$(I-P_r)(I-Q_m)=I-P$ and therefore
\[
 (T_{\rm op})_{\rm reg}=(I-P_r)T_{\rm op}=(I-P_r)(I-Q_m)T=(I-P)T=T_{\rm reg}.
\]
This proves the first statement in \eqref{Tm2}.
Similarly, $P_r(I-Q_m)=P_r$ implies that
\[
 (T_{\rm op})_{\rm sing}=P_rT_{\rm op}=P_r(I-Q_m)T=P_rT.
\]
The identity $P=P_r+Q_m$ now gives the stated formula for the
singular part $T_{\rm sing}$:
\[
 T_{\rm sing}=(P_r+Q_m)T=P_rT+ T_{\rm mul}=(T_{\rm op})_{\rm sing} + T_{\rm mul}.
\]

The statement concerning the uniqueness of Lebesgue type
decompositions of $T$ and $T_{\rm op}$ follows immediately from the
equality of the regular parts in \eqref{Tm2} by Theorem \ref{uninew}, since
$T_{\rm reg}$ is bounded if and only if  $(T_{\rm op})_{\rm reg}$
is bounded.
\end{proof}

\begin{definition}\label{defwLebtype}
Let $T$  be a linear relation from a Hilbert space $\sH$ to a
Hilbert space $\sK$. Then $T$ is said to have a \textit{weak
Lebesgue type decomposition} if it admits an orthogonal range
decomposition $T=T_1+T_2$, where $T_1$ is an operator and $T_2$ is
singular.
\end{definition}

Next to the above existence argument there is an enumeration of all
weak Lebesgue type decompositions of $T$.

\begin{theorem}\label{wLebtype}
Let $T$  be a relation from a Hilbert space
$\sH$ to a Hilbert space $\sK$ and let
$\sM$ be closed linear subspace in $\sK$ such that
\begin{equation}\label{wMleb}
    \cmul T \subset \sM,
   \quad
   \sM \cap \dom T^* \subset \ker T^*.
\end{equation}
Then the relations $(I-P_{\sM})T$ and $P_{\sM}T$ satisfy
\begin{enumerate}[ $ (\rm i)$]
\item $(I-P_{\sM})T$ is an operator,
 \item $( P_{\sM}T )^{**}
=\cdom T^{**}\times (\sM \cap (\sM \cap \dom T^{*})^{\perp})$,
i.e., $P_{\sM}T$ is singular,
\end{enumerate}
and $T$ has a weak Lebesgue type decomposition of the form
\begin{equation}\label{wlebdeco}
T = (I-P_{\sM}) T + P_{\sM}T.
\end{equation}
Conversely, if $T$ has a weak Lebesgue type decomposition
$T=T_1+T_2$ as in Definition~\ref{defwLebtype}, then the subspace
$\sM=\cran T_2\,(=\mul T_2^{**})$ satisfies \eqref{wMleb} and
generates this decomposition via \eqref{wlebdeco}: $T_1=(I-P_{\sM})
T$ and $T_2=P_{\sM}T$.
\end{theorem}

\begin{proof}
The statements (i) and (ii) follow from Lemma \ref{Lebtypelem} while
the decomposition \eqref{wlebdeco} is obtained from
Corollary~\ref{osdd} (ii).

For the converse part one can apply (the proof of) Lemma~\ref{osd}
to get the representations for $T_1$ and $T_2$ with the choice
$\sM=\cran T_2$. Then the properties in \eqref{wMleb} follow again
from Lemma \ref{Lebtypelem}.
\end{proof}

Analogous to Theorem \ref{Lebtype2} there is a one-to-one
parametrization of all weak Lebesgue type decomposition of $T$.

\begin{theorem}\label{wwLebtype2}
Let $T$ be a relation from the Hilbert space $\sH$ to the Hilbert
space $\sK$. Then the weak Lebesgue type decompositions of $T$ are
in one-to-one correspondence via $\sM=\cran T_2=\cmul T \oplus \sL$
with the closed linear subspaces $\sL \subset (\cmul T)^\perp$ which
satisfy the condition
\begin{equation}\label{wwLleb2}
 \sL\cap \dom T^*=\{0\}.
\end{equation}
In particular, in this case
\begin{equation}\label{wwLleb2+}
 (T_2)^{**} =\cdom T^{**}\times (\cmul T \oplus \sL).
\end{equation}
\end{theorem}

\begin{proof}
For the representation of $\sM$ use the condition $\cmul T \subset
\sM$ in \eqref{wMleb} to decompose $\sM=\cmul T \oplus \sL$ with
some closed subspace $\sL$. This decomposition of $\sM$ combined
with the second condition in \eqref{wMleb} leads to
\begin{equation}\label{wLsing1}
 \sM \cap \dom T^*=\sL \cap \dom T^*\subset \ker T^*,
\end{equation}
which means that $T_2:=P_\sM T$ is singular. On the other hand,
\[
 (\ran T_2)^\perp=\ker T^*P_\sM=\ker P_\sM \oplus(\ran P_\sM\cap
\ker T^*).
\]
By taking orthogonal complements in this identity one concludes that
\begin{equation}\label{wLsing2}
 \sM=\cran T_2 \quad\Leftrightarrow\quad \sM\cap \ker
 T^*=\{0\}.
\end{equation}
One concludes from the formulas \eqref{wLsing1} and \eqref{wLsing2}
that the choice $\sM=\cran T_2$ is equivalent to $\sL \cap \dom
T^*=\{0\}$. Now by Theorem \ref{wLebtype} all weak Lebesgue type
decompositions of $T$ can be described via $\sM=\cmul T \oplus
\sL=\cran T_2$ with some subspace $\sL$ satisfying \eqref{wwLleb2}.
Moreover, since $T_2$ is singular, one obtains \eqref{wwLleb2+} from
item (ii) of Theorem \ref{wLebtype}. The formula \eqref{wwLleb2+}
shows that the correspondence between the singular component $T_2$
and the closed subspace $\sL$ is bijective.
\end{proof}

Furthermore, there is an analog of Ando's uniqueness criterion for
weak Lebesgue type decompositions.

\begin{theorem}\label{uniOp}
Let $T$ be a relation from the Hilbert space $\sH$ to the Hilbert
space $\sK$. Then the following conditions are equivalent:

\begin{enumerate}[ $(\rm i)$]

\item
the weak Lebesgue decomposition is the only weak Lebesgue
type decomposition of $T$;

\item $\dom T^*=(\mul T)^{\perp}$;

\item $T_{\rm op}$ is bounded.
\end{enumerate}
In this case $ T_{\rm op}=T_{\rm reg}$ and the weak Lebesgue
decomposition coincides with the Lebesgue decomposition of $T$.
\end{theorem}

\begin{proof}
(i) $\Leftrightarrow$ (ii) By Theorem \ref{wwLebtype2} the weak
Lebesgue type decompositions of $T$ are in one-to-one correspondence
via $\sM=\cran T_2=\cmul T \oplus \sL$ with the closed subspaces
$\sL \subset (\cmul T)^\perp$ which satisfy the condition
\eqref{wwLleb2}. Since $\dom T^*\subset (\cmul T)^\perp$, one
concludes that $\sL=\{0\}$ is the only subspace of $(\cmul T)^\perp$
satisfying the condition \eqref{wwLleb2} precisely when the equality
$\dom T^*=(\mul T)^\perp$ holds. Clearly, the choice $\sL=\{0\}$
corresponds to the weak Lebesgue decomposition in Theorem
\ref{Tmlemma}.

(ii) $\Leftrightarrow$ (iii) By Proposition~\ref{stoch0} (see items
(iv), (vi) therein) $T_{\rm op}=(I-P_m)T$ or, equivalently, $(T_{\rm
op})^{**}$, is a bounded operator if and only if
\begin{equation}\label{Topbounded}
\sK=\dom (T_{\rm op})^{*}=\dom T^*(I-P_m)=\ran P_m\oplus
\ran(I-P_m)\cap \dom T^*.
\end{equation} Since $\dom T^*\subset
(\mul T)^\perp=\ran (I-P_m)$ the condition in \eqref{Topbounded} is
equivalent to the condition $\dom T^*=(\mul T)^\perp$ in (ii).

As to the last statement observe that, since in a Lebesgue type
decomposition \eqref{lebdeco} in Theorem \ref{Lebtype} $(I-P_{\sM})
T$ is a closable operator and $P_{\sM}T$ is singular, the uniqueness
assumption implies that there is only one such decomposition which
then necessarily coincides the Lebesgue decomposition of $T$; cf.
Theorem~\ref{uninew}. In particular, the equality $T_{\rm op}=T_{\rm
reg}$ must hold. This completes the proof.
\end{proof}

The conditions on the subspace $\sL$ in Theorem \ref{wwLebtype2} are
essentially weaker than the conditions appearing in Theorem
\ref{Lebtype2}. Hence, in general the class of all weak Lebesgue
type decompositions is much wider than the class of all Lebesgue
type decompositions of $T$. In particular, if $T$ is an operator
then the class of all weak Lebesgue type decompositions is simply
parametrized by the subspaces
$\sM=\sL$ such that $\sM\cap \dom T^*=\{0\}$.\\

Finally, notice that an analog of Theorem \ref{ExtUnirep} does not
hold for weak Lebesgue type decompositions. In particular, if $T_1$
in the decomposition $T=T_1+T_2$ is a bounded operator, it does not
follow that $T_{\rm op}$ is bounded. As an example consider any
non-closable operator $T$, whose regular part $T_{\rm reg}$ is
bounded. Since $T$ is an operator, one has $T_{\rm op}=T$ and this
is unbounded because $T$ is not closable.

\section{Domination and closability}\label{sec8}

It has been shown in the previous sections that the (weak) Lebesgue
type decompositions are in general non-unique; in particular, this
is the case for any unbounded operator $T$ which always admits
infinitely many (weak) Lebesgue type decompositions. In this section
it is shown that there are specific domination properties for the
components $T_{\rm reg}$ and $T_{\rm op}$ which characterize the
unique (weak) Lebesgue decomposition of $T$ among all (weak)
Lebesgue type decompositions $T=T_1+T_2$. Moreover, as a further
result it is shown that the notion of domination can be used to
derive a closability criterion for operators in Hilbert spaces.

\subsection{Domination and maximality properties}
Linear relations  have a preorder which can be introduced by means
of the notion of domination, which in its present general form was
introduced and studied in \cite{HSn2015}.

\begin{definition}\label{domi}
Let $S_{1}$ and $S_{2}$ be a linear relations (not necessarily
closed) from a Hilbert space $\sH$ to Hilbert spaces $\sK_{1}$ and
$\sK_{2}$, respectively. Then $S_{1}$ is said to be
\textit{dominated} by $S_{2}$, denoted by
\begin{equation}\label{prec}
 S_{1} \prec S_{2},
\end{equation}
if there exists an operator $C \in \bB(\sK_{2}, \sK_{1})$ such that
$CS_{2} \subset S_{1}$. The domination is said to be contractive if
$C$ is a contraction.
\end{definition}

Clearly, domination is symmetric: $S\prec S$, and domination is also transitive: if
$C_1S_2\subset S_1$ and $C_2 S_3\subset S_2$ then
\[
 C_1(C_2S_3)\subset C_1 S_2\subset S_1,
\]
so that $(C_1C_2)S_3\subset S_1$. Hence $\prec$ defines a preorder
in the class of linear relations. However, the relation $\prec$ is
in general not antisymmetric, even when $\prec$ is contractive (or
isometric).

In the particular case that both $S_{1}$ and $S_{2}$ are operators
domination takes the following form; see  \cite{HSn2015}.

\begin{lemma}\label{domlemma}
Let $S_{1}$ and $S_{2}$ be linear operators from a Hilbert space
$\sH$ to Hilbert spaces $\sK_{1}$ and $\sK_{2}$, respectively.
 Then there is a constant $c>0$ such that
\begin{equation}\label{dominate}
\dom S_2\subset \dom S_1 \quad \text{and} \quad \|S_1f\|\le
c\|S_2f\|, \quad f\in \dom S_2,
\end{equation}
if and only if \,$S_{1}$ is dominated by $S_{2}$ with a bounded
operator $C\in\bB(\sK_2,\sK_1)$ with $\|C\| \leq c$ such that
\begin{equation}\label{include}
 C S_2 \subset S_1.
\end{equation}
The domination is contractive if and only if \eqref{dominate} holds
for some $c \leq 1$.
\end{lemma}

\begin{proof}
Assume that \eqref{dominate} holds and define an operator $C_0$ from
$\ran S_2$ to $\ran S_1$ by $C_0S_2f=S_1f$, $f\in \dom S_2$. It
follows from \eqref{dominate} that $C_0$ is well defined and bounded
with $\|C_0\|\leq c$. Thus $C_0$ can be continued to a bounded
operator from $\cran S_2$ to $\cran S_1$. Let $C$ be the extension
of the closure $C_0^{**}$ obtained by defining $C$ to be $0$ on
$(\ran S_2)^\perp$. Then $C \in \bB(\sK_{2}, \sK_{1})$, $\|C\| \leq
c$, and $CS_2\subset S_1$ holds.

Conversely, assume that \eqref{include} holds for some
$C\in\bB(\sK_{2}, \sK_{1})$ with $\|C\| \leq c$. Then clearly
\[
 \dom S_2\subset \dom S_1 \quad\text{and}\quad \|S_1f\| \le
\|C\|\,\|S_2f\|, \quad f\in\dom S_2,
\]
and \eqref{dominate} is satisfied.
\end{proof}

For some further connections between domination, Douglas type
factorizations, and range inclusions, see \cite{HSn2015}. Notice
that Douglas type factorizations and range inclusions for linear
relations were recently established in D. Popovici and Z.
Sebesty\'en \cite{PS} in the context of linear spaces.\\

Now let $T$ be a linear relation from  $\sH$ to $\sK$. Then the
regular part $T_{\rm reg}=(I-P)T$ in the Lebegue decomposition
\eqref{resi} is a closable operator which, by definition, is
contractively dominated by $T$:
\begin{equation}\label{domin}
 T_{\rm reg} \prec T.
\end{equation}
Similarly, the operator $T_{\rm op}=(I-P_{m})$ in the weak Lebesgue
decomposition \eqref{decoms0} is contractively dominated by $T$:
\begin{equation}\label{dominOp}
 T_{\rm op} \prec T.
\end{equation}
Due to $T_{\rm reg}=(I-P)T_{\rm op}$, cf. \eqref{regg}, one also sees that
$T_{\rm reg} \prec T_{\rm op}$ contractively, i.e.,
\[
\|T_{\rm reg}f\|\le \|T_{\rm op}f\|,  \quad f\in \dom T.
\]
In the next theorem it is shown that the regular part $T_{\rm reg}$ and the
operator part $T_{\rm op}$ possess certain maximality properties in the
sense of domination.

\begin{theorem}\label{regmax}
Let $S$ and $T$ be linear relations from a Hilbert space $\sH$
to a Hilbert space $\sK$.
Then domination is preserved for regular parts:
\begin{equation}\label{ass0a}
 S \prec T \quad\Rightarrow\quad S_{\rm reg}\prec T_{\rm reg},
\end{equation}
and also contractive domination is preserved.
In particular,  $T_{\rm reg}$ is a maximal closable operator that is
dominated by $T$:
\begin{equation}\label{ass0b}
 S=S_{\rm reg}\prec T \quad\Rightarrow\quad S \prec T_{\rm reg}.
\end{equation}
Similarly, domination is preserved for operator parts:
\begin{equation}\label{ass0a2}
 S \prec T \quad\Rightarrow\quad S_{\rm op}\prec T_{\rm op},
\end{equation}
and also contractive domination is preserved. In particular, $T_{\rm
op}$ is a maximal operator that is dominated by $T$:
\begin{equation}\label{ass0b2}
 S=S_{\rm op}\prec T \quad\Rightarrow\quad S \prec T_{\rm op}.
\end{equation}
\end{theorem}

\begin{proof}
Since $S \prec T$ there is an operator $C \in \bB(\sK)$ such that
$CT \subset S$. Now by \eqref{prod1} one sees that
$S^{*} \subset T^{*}C^{*}$, and again by \eqref{prod1}:
\[
   CT^{**} \subset (T^*C^{*})^* \subset S^{**}.
\]
In particular, this shows that $C$ maps $\mul T^{**}$ into $\mul
S^{**}$. Let $P$ be the orthogonal projection onto $\mul T^{**}$ and
let $R$ be the orthogonal projection onto $\mul S^{**}$. Write
$\{f,g\}=\{f, (I-P)g+Pg\} \in T$. Here  $Pg \in \mul T^{**}$ and one
concludes that
\[
 \{f,Cg\} = \{f, C(I-P)g+CPg\} \in S,
\]
where  $CPg \in \mul S^{**}$.
Hence it follows that
\[
\{f, (I-P)g\} \in T_{\rm reg} \quad \Rightarrow \quad  \{f,
(I-R)C(I-P)g\} \in S_{\rm reg}.
\]
Equivalently, $[(I-R)C ] T_{\rm reg} \subset S_{\rm reg}$, which
implies $S_{\rm reg}\prec T_{\rm reg}$. In addition, observe that if
$C$ is a contraction then also $(I-R)C$ is a contraction.

By \eqref{domin} one sees that $T_{\rm reg}$ is a closable operator
which is dominated by $T$. Now let $S=S_{\rm reg}$ be any closable
operator which is dominated by $T$: $S  \prec T$. Then \eqref{ass0a}
shows that $S=S_{\rm reg}  \prec T_{\rm reg}$.

Thus all the statements concerning the regular part are proven.

Now consider the operator parts $T_{\rm op}$ and $S_{\rm op}$. Let
$Q_m$ be the orthogonal projection from $\sK$ onto $\cmul T$ and let
$R_m$ be the orthogonal projection from $\sK$ onto $\cmul S$. Then
$CT \subset S$ shows that $C$ maps $\mul T$ into $\mul S$ and hence
also $\cmul T$ into $\cmul S$. Now replacing $P$ by $Q_m$ and $R$ by
$R_m$ in the first part of the proof shows that
\[
\{f, (I-Q_m)g\} \in T_{\rm op} \quad \Rightarrow \quad  \{f,
(I-R_m)C(I-Q_m)g\} \in S_{\rm op}.
\]
Equivalently, $[(I-R_m)C ] T_{\rm op} \subset S_{\rm op}$, which
implies $S_{\rm op}\prec T_{\rm op}$. In addition, if $C$ is a
contraction then also $(I-R_m)C$ is a contraction.

By \eqref{dominOp} $T_{\rm op}\prec T$ and if $S$ is an operator,
then $S=S_{\rm op}$. Hence, if $S$ is dominated by $T$, then
\eqref{ass0a2} shows that $S=S_{\rm op}  \prec T_{\rm op}$.
\end{proof}

Theorem \ref{regmax} singles out optimality of the Lebesgue
decomposition of $T$ in Theorem \ref{jor0} among all Lebesgue type
decompositions of $T$ in Theorem \ref{Lebtype2}.

\begin{corollary}\label{optim}
Let $T$ be a linear relation from a Hilbert space $\sH$ to a Hilbert
space $\sK$. Let $T=T_1+T_2$ be any Lebesgue type decomposition of
$T$, where $T_{1}$ is regular and $T_{2}$ is singular. Then
\[
T_1 \prec T_{\rm reg}.
\]
\end{corollary}

\begin{proof}
If $T=T_1+T_2$ is a Lebesgue type decompositions of $T$ then there
exists some orthogonal projection $P_{\sM}$ such that
$T_{1}=(I-P_{\sM})T$, which shows that the operator $T_{1}$ is
dominated by $T$: $T_{1} \prec T$. Since $T_{1}$ is regular one may
apply Theorem \ref{regmax} and obtain $T_{1} \prec T_{\rm reg}$.
\end{proof}

There is a similar optimality property for the weak Lebesgue
decomposition $T=T_{\rm op}+T_{\rm mul}$ in Theorem \ref{Tmlemma}
among all weak Lebesgue type decompositions described in Theorem
\ref{wLebtype}.

\begin{corollary}\label{optimOp}
Let $T$ be a linear relation from a Hilbert space $\sH$ to a Hilbert
space $\sK$. Let $T=T_1+T_2$ be any weak Lebesgue type decomposition
of $T$, where $T_1$ is an operator and $T_{2}$ is singular. Then
\[
 T_1 \prec T_{\rm op}.
\]
\end{corollary}

\begin{proof}
If $T=T_1+T_2$ is a weak Lebesgue type decompositions of $T$ with
$T_1$ an operator, then by Lemma \ref{osd} there exists an
orthogonal projection $R$ such that $T_{1}=(I-R)T$. Hence, $T_{1}
\prec T$ and, since $T_{1}$ is an operator, an application of
\eqref{ass0b2} yields $T_{1} \prec T_{\rm op}$.
\end{proof}

\begin{remark}\label{maximality}
The optimality properties stated in the above corollaries for
$T_{\rm reg}$ and $T_{\rm op}$ hold in a slightly more general form.
Indeed, (the proofs of) the stated optimality properties do not use
any singularity assumption on the relation $T_2$. For instance, for
any distinguished orthogonal range decomposition $T=T_1+T_2$ of $T$
one has:
\[
 T_1 \prec T_{\rm op}.
\]
\end{remark}

Note that in Corollary \ref{optim} $T_{1} \prec T$ contractively and
hence also $T_{1} \prec T_{\rm reg}$ contractively. Since $T_{1}$
and $T_{\rm reg}$ are operators, Lemma \ref{domlemma} shows that
\[
\|T_1f\|\le  \|T_{\rm reg}f\|, \quad f \in \dom T.
\]
The maximality property of $T_{\rm reg}$ reflects the following
inequality for the associated projectors $(I-P)\geq (I-P_\sM)$,
which is clear from $\mul T^{**}\subset \sM$; cf. \eqref{Mleb11}.
Similarly, when $T_1$ in $T=T_1+T_2$ is an operator, one has $\mul
T\subset \ker (I-R)=\ran R$. Hence, $R\geq Q_m$ and $I-Q_m\geq I-R$,
which reflects the maximality property of $T_{\rm op}$ in Corollary
\ref{optimOp}. It should be mentioned that for a densely defined
operator $T$ an equivalent minimality property of the projection $P$
appears in \cite{Ota3}: $P\leq P_\sM$ when $T_2=P_\sM T$ is a
singular part
of $T$.\\

The maximality properties of $T_{\rm reg}$ and $T_{\rm op}$ imply
that the Lebesgue decomposition and the weak Lebesgue decomposition
are preserved under unitary similarity transforms; in fact these
results hold in the following slightly more general form.

\begin{proposition}\label{LebEquivalent}
Let $\sH_1$, $\sH_2$, $\sK_1$, and $\sK_2$ be Hilbert spaces and let
$T: \sH_1 \to \sK_1$ and $\wt T:\sH_2 \to \sK_2$ be linear relations
which are equivalent via two unitary operators $V:\sH_1\to\sH_2$ and
$U:\sK_1\to\sK_2$ in the sense that
\begin{equation}\label{simi1}
 \wt T = U T V.
\end{equation}
Then their Lebesgue decompositions $T=T_{\rm{reg}}+T_{\rm{sing}}$
and $\wt T=\wt T_{\rm{reg}}+ \wt T_{\rm{sing}}$ are also equivalent
via the same unitary operators $U$ and $V$:
\begin{equation}\label{simi2}
 \wt T_{\rm{reg}}=U T_{\rm{reg}}V \quad\text{and}\quad \wt T_{\rm{sing}}=U T_{\rm{sing}}V.
\end{equation}
Similarly their weak Lebesgue decompositions
$T=T_{\rm{op}}+T_{\rm{mul}}$ and $\wt T=\wt T_{\rm{op}}+ \wt
T_{\rm{mul}}$ are also equivalent via the same unitary operators $U$
and $V$:
\begin{equation}\label{simi2Op}
 \wt T_{\rm{op}}=U T_{\rm{op}}V \quad\text{and}\quad \wt T_{\rm{mul}}=U T_{\rm{mul}}V.
\end{equation}
\end{proposition}

\begin{proof}
Let $P$ and $\wt P$ be the canonical projections corresponding to
the Lebesgue decompositions of $T$ and $\wt T$: $T=(I-P)T+PT$ and
$\wt T=(I-\wt P)\wt T+\wt P \wt T$. Then
\begin{equation}\label{simi3}
 (I-\wt P)\wt T+\wt P\wt T=U[(I-P)T+PT]V=U(I-P)U^*\wt T+UPU^*\wt T.
\end{equation}
Here $UPU^*$ is an orthogonal projection and clearly $U(I-P)U^*\wt
T=U(I-P)TV$ is regular and $UPU^* \wt T=U(PT)V$ is singular. Hence,
the righthand side of \eqref{simi3} is a Lebesgue type decomposition
of $\wt T$. By the maximality property of the $\wt T_{\rm{reg}}$ and
$I-\wt P$, see Remark~\ref{maximality}, one concludes that $(I-\wt
P)\ge U(I-P)U^*$.

On the other hand, by applying the above conclusion to $T = U^* \wt
T V^*$ one gets the inequality $(I-P)\ge U^*(I-\wt P)U$, or
equivalently, $U(I-P)U^*\ge (I-\wt P)$. Therefore, $(I-\wt
P)=U(I-P)U^*$ and $\wt P=UPU^*$ and, consequently,
\[
 (I-\wt P)\wt T=U(I-P)TV, \quad
 \wt P \wt T=UPTV,
\]
which proves \eqref{simi2}. The proof of \eqref{simi2Op} is similar.
\end{proof}

\subsection{A criterion for closability}
The following two theorems show that closability can be
characterized in terms of a sequence of operators which are
successively contractively dominated. The first theorem guarantees
closability.

\begin{theorem}\label{thmcl}
Let $T$ be a linear operator from a Hilbert space $\sH$ to a Hilbert space $\sK$
and assume that there exists a sequence of operators
$T_n\in \bB(\cdom T, \sK_{n})$ such that
 \begin{equation}\label{amono+}
 \|T_m f\| \le \|T_n f\|, \quad  f\in \cdom T, \quad m \leq n,
\end{equation}
and
\begin{equation}\label{ass0}
   \|T_{n}f\|  \nearrow \|Tf\|, \quad f \in \dom T.
\end{equation}
Then the operator $T$ is closable.
\end{theorem}

\begin{proof}
 Assume that $T$ is a linear operator from $\sH$ to $\sK$
and let $T_{\rm{reg}}$ be the regular part of $T$ as defined in
Theorem \ref{jor0}. By \eqref{ass0} one has $\|T_n f\| \le \|Tf\|$
for all $f \in \dom T$. According to Lemma~\ref{domlemma} there
exists a contraction $C_n\in\bB(\sK,\sK_{n})$ such that $C_n T
\subset T_n$ for all $n\in\dN$. By \eqref{prod1} this implies that
$(T_{n})^{*} \subset T^{*}(C_{n})^{*}$, and then again by
\eqref{prod1} one obtains
\begin{equation}\label{CT**}
   C_n T^{**} \subset (T^*C_n)^*  \subset T_n.
\end{equation}
In particular, if $\{0,\varphi\} \in T^{**}$, then $\{0,C_{n} \varphi\} \in T_{n}$,
so that $C_{n}\varphi=0$.
Thus one concludes that $\mul T^{**}\subset \ker C_n$.
Now, let $P$ be the orthogonal projection from $\sK$ onto $\mul
T^{**}$. Then $C_nP=0$ which combined with Theorem \ref{jor0} leads
to
\[
 C_n T=C_n[(I-P)T+PT]=C_n(I-P)T= C_n T_{\rm reg},
\]
Hence, $C_n T_{\rm reg} \subset T_n$ for all $n\in\dN$ and by
Lemma~\ref{domlemma} this implies that
\begin{equation}\label{ass2}
 \|T_n f\| \le \|T_{\rm{reg}}f\| \le \|Tf\|, \quad f \in \dom T.
\end{equation}
Via \eqref{amono+} and \eqref{ass0} one may take the supremum over
$n \in \dN$ in \eqref{ass2} to obtain the equality
\[
 \|Tf\|=\|T_{\rm{reg}}f\|, \quad f\in\dom T.
\]
This implies that $T_{\rm{sing}}=0$ and hence $T$ is closable, cf.
Corollary \ref{CorRegSing}.
\end{proof}

The second theorem is a converse to Theorem \ref{thmcl+}:
each closable operator can be approximated by a sequence
as in \eqref{amono+} and \eqref{ass0}.

\begin{theorem}\label{thmcl+}
Let $T$ be a closable linear operator from a Hilbert space
$\sH$ to a Hilbert space $\sK$.
Then there exists a sequence $T_n \in \bB(\cdom T,\sH)$
with the properties \eqref{amono+} and \eqref{ass0}.
\end{theorem}

\begin{proof}
Assume that $T$ is a closable operator from $\sH$ to $\sK$. Then
$T^*T^{**}$ is a nonnegative selfadjoint relation in $\sH$ with
$\dom (T^*T^{**})^{1/2}=\dom T^{**}\supset \dom T$, see
Section~\ref{sec2.2}. Let $(T^*T^{**})_{\rm{s}}$ be the orthogonal
operator part of $T^*T^{**}$ and let
\[
(T^*T^{**})_{\rm{s}}^{1/2}=\int_0^\infty \lambda\,dE_\lambda
\]
be the spectral representation of $(T^*T^{**})_{\rm{s}}^{1/2}$. Note
that $\dom (T^*T^{**})_{\rm{s}}^{1/2}$ is dense in $\cdom T$ and the
orthogonal projections $E_\lambda \in \bB(\cdom T, \sH)$. By means
of the corresponding spectral family, define the sequence of
selfadjoint operators $T_n\in \bB(\cdom T,\sH)$ by
 \[
 T_n=\int_0^n \lambda\,dE_\lambda, \quad n\in\dN.
\]
Then clearly $\|T_mf\|\le \|T_nf\|$ for $m\le n$, which gives
\eqref{amono+}. Observe, that if $P$ is the orthogonal projection
onto the closed linear subspace
\[
\cdom T=\cdom T^{**} =\cdom (T^*T^{**})_{\rm{s}},
\]
then $PT^*=T_{\rm{s}}$ is the operator part of $T^*$ and similarly
$PT^*T^{**}=(T^*T^{**})_{\rm{s}}$. The closure of the form
$(Tf,Tg)$, $f,g\in\dom T$, is given by
\[
 (T^{**}f,T^{**}g)=((T^*T^{**})_{\rm{s}}^{1/2}f,(T^*T^{**})_{\rm{s}}^{1/2}g),
 \quad
 f,g\in\dom T^{**},
\]
see e.g. \cite{Kato}, \cite{HSn2015}. Now by the construction of the
sequence $T_n$ one obtains
\[
 \|T_nf\| \nearrow \|(T^*T^{**})_{\rm{s}}^{1/2}f\|
 =\|T^{**}f\|, \quad f\in \dom T^{**}.
\]
Therefore, \eqref{ass0} has been shown and this completes the proof.
\end{proof}

The convergence of the operators $T_n$ to the operator $T$ in
Theorem \ref{thmcl}  and Theorem \ref{thmcl+} is in fact formulated
in terms of contractive domination. Notice that this kind of
convergence is closely related to the convergence of the sequence of
forms $\st_N(f,g):=(T_n f,T_n g)$ and in that way it can also be
related to some other notions of convergence, like graph and strong
resolvent convergence; see \cite{BHSW}.

\subsection{A metric characterization of closable operators}

Let $T$ be a relation from a Hilbert space $\sH$ to a Hilbert space
$\sK$. Then the linear relation $(I-P)T^{**}$ where $P$ is the
orthogonal projection onto $\mul T^{**}$ has a useful metric
property; cf. \cite{HSeSnSz}. For the convenience of the reader a
complete proof is provided.

\begin{lemma}\label{earlier}
Let $T$ be a relation from a Hilbert space $\sH$ to a Hilbert space
$\sK$ and let $P$ be the orthogonal projection from $\sK$ onto $\mul
T^{**}$. Then  for all $\{f,f'\} \in T^{**}$:
\begin{equation}\label{drei+}
\|(I-P)f'\|^{2}=\sup_{h \in \dom T}
\left\{  \inf_{\{g,g'\} \in T} \left\{\|g'\|^2+\|h-g\|^2\right\} - \|f-h\|^2  \right\}.
\end{equation}
\end{lemma}

\begin{proof}
Provide the graph of $T^{**}$ with the graph norm $\|\cdot \|_{\overline{T}}$.
It is clear that the
graph of $T$ is dense in $T^{**}$ in this sense. The Hilbert space
$T^{**}$ is now written as an orthogonal sum of two closed linear subspaces:
\begin{equation}\label{graph}
 T^{**}=(T^{**} \hominus (\{0\} \times \mul T^{**})) \hoplus (\{0\} \times \mul T^{**})
\end{equation}
Let $Q$ be the orthogonal projection onto $\{0\} \times \mul T^{**}$.

\textit{Step 1}. Let $\{f,f'\} \in T^{**}$ and
define the elements $\{h,h'\}=(I-Q)\{f,f'\}$
and $\{0,k'\}=Q\{f,f'\}$.
First note that by definition for all $t' \in \mul T^{**}$
\[
 0=(\{h,h'\}, \{0, t'\})_{\overline{T}} =(h', t'),
 \]
which gives $h' \in (\mul T^{**})^{\perp}$.
Therefore it follows from the decomposition
\[
 \{f,f'\}=\{h,h'\}+\{0, k'\},
 \quad \{h,h'\} \in T^{**} \hominus (\{0\} \times \mul T^{**}),
 \quad k' \in \mul T^{**},
\]
that $f'=h'+k'$ and $k'=P f'$.
Hence one concludes that $Q\{f,f'\} =\{0, P f'\}$ and
\begin{equation}\label{vijf}
\|Q\{f,f'\} \|_{\overline{T}}^{2} = \|\{0,Pf'\}
\|_{\overline{T}}^{2} = \|Pf'\|^{2}.
\end{equation}

\textit{Step 2}. Introduce the linear mapping $\imath :T^{**} \to \cdom T^{**}$ by
\[
 \imath\{f,f'\}=f, \quad \{f,f'\} \in T^{**}.
\]
Then clearly $\iota$ is a contraction from the Hilbert space
$T^{**}$ to the Hilbert space $\cdom T^{**}$ with $\ran \imath= \dom
T^{**}$ is dense in $\cdom T^{**}$ and $\ker \imath=\{0\}\times\mul
T^{**}$.
The adjoint $\imath^{*}$ maps $\cdom T^{**}$ into $T^{**} \hominus
(\{0\} \times \mul T^{**})$ and $\imath^{*}(\dom T)$ is dense in
$T^{**} \hominus (\{0\} \times \mul T^{**})$, since $\dom T$ is
dense in $\cdom T^{**}$.

\textit{Step 3}. Let $\{f, f'\} \in T^{**}$.
Then it follows from the orthogonal decomposition \eqref{graph} that
 \[
 \|Q \{f,f'\}\|_{\overline{T}}^2= \inf \left\{\,
\|\{f,f'\}-\{l,l'\}\|_{\overline{T}}^2:\,
\{l,\l'\} \in T^{**} \hominus (\{0\} \times \mul T^{**})
 \,\right\}.
\]
However, since $ \imath^*(\dom T)$ is dense in
 $T^{**} \hominus (\{0\} \times \mul T^{**})$,
it follows that
\begin{equation}
\label{einz}
\begin{split}
 \left\|Q \{f,f'\}\right\|^2_{\overline{T}} &
 = \inf_{h \in \dom T} \|\{f,f'\}-\imath^* h\|_{\overline{T}}^{2} \\
 &= \inf_{h \in \dom T}
\left\{(f,f)+(f',f')-(f,h)-(h,f)
+\left( \imath^* h, \imath^*h\right)_{\overline{T}} \right\}
\\
&=\|f'\|^2+\inf_{h \in \dom T}
\left\{ \|f-h\|^2-\|h\|^2 +\left\| \imath^* h \right\|^2_{\overline{T}}
\right\}.
\end{split}
\end{equation}

\textit{Step 4}. Since $T$ is dense in $T^{**}$, every
element of the form $\imath^* h$, $h \in \dom T$, can be
approximated by elements in $T$, which leads to
\begin{equation}
\label{zwei}
\begin{split}
 0&= \inf_{\{g,g'\} \in T}
 \left\{ \left\|\imath^* h-\{g,g'\} \right\|^2_{\overline{T}}  \right\} \\
  &= \left\| \imath^* h \right\|^2_{\overline{T}}+\inf_{\{g,g'\}
\in T} \left\{
(g,g)+(g',g')-(g,h)-(h,g) \right\}     \\
  &= -\|h\|^2 + \left\| \imath^* h \right\|^2_{\overline{T}}
+\inf_{\{g,g'\} \in T}
  \left\{\|g'\|^2+\|h-g\|^2\right\}.
\end{split}
\end{equation}
\textit{Step 5}.
Combine the identities \eqref{einz}, \eqref{zwei}, and
\eqref{vijf} to get
\[
  \|Pf'\|^{2}=\|f'\|^2+\inf_{h \in \dom T} \left\{ \|f-h\|^2
-\inf_{\{g,g'\} \in T} \left\{\|g'\|^2+\|h-g\|^2\right\}  \right\},
\]
which gives \eqref{drei+}.
\end{proof}

The following characterization of closable operators
is now a straightforward consequence of Lemma \ref{earlier}.
It may be seen as an alternative to Theorem \ref{thmcl}
and Theorem \ref{thmcl+}.

\begin{theorem}\label{dik}
Let $T$ be a linear relation from a Hilbert space $\sH$ to a Hilbert
space $\sK$. Then $T$ is regular if and only if for all $\{f, f'\}
\in T$:
\begin{equation}\label{vierv}
\|f' \|^2=
\sup_{h \in \dom T}
\left\{  \inf_{\{g,g'\} \in T} \left\{\|g'\|^2+\|h-g\|^2\right\} - \|f-h\|^2  \right\},
\end{equation}
\end{theorem}

\begin{proof}
Recall from Corollary \ref{CorRegSing} that
$T$ is regular if and only if $T=T_{\rm reg}$,
which is equivalent to $T_{\rm sing}=0$.
Furthermore, this last condition is the same as
\[
 Pf'=0 \quad \mbox{for all} \quad \{f,f'\} \in T.
\]
An application of Theorem \ref{earlier} completes the proof.
\end{proof}

Another situation in which Lemma \ref{earlier} can be used
is when the relation is a range space relation. Then there are
useful applications
in terms of parallel sums and differences for pairs of
bounded linear operators; cf. \cite{HSn2018}.

\end{document}